\documentclass{amsart}
\usepackage{tikz}
\usepackage{array}
\usepackage{amsfonts}
\usepackage{amsmath}
\usepackage{amssymb}
\usepackage{amsthm}

\usepackage{bm}
\usepackage{MnSymbol}
\usepackage{physics}
\usepackage{algorithm}
\usepackage{algorithmic}
\usepackage[utf8]{inputenc}
\usepackage{amsmath, amsthm, amscd, amssymb, amsfonts, latexsym,amsrefs}
\usepackage{fullpage}
\usepackage{bm}
\usepackage[all]{xy}
\usepackage{tikz}
\usepackage{fancybox}
\usepackage{hyperref}
\usepackage{siunitx}

\usepackage{enumerate}

\DeclareMathOperator{\re}{Re}
\DeclareMathOperator{\im}{Im}

\newcommand{\kommentar}[1]{}
\newcommand{\ord}{\text{ord}}

\newcommand{\F}{\mathbb F}

\newcommand{\Z}{\mathbb Z}
\newcommand{\Q}{\mathbb Q}
\newcommand{\R}{\mathbb R}

\newcommand{\C}{\mathbb C}

\newtheorem{lem}{Lemma}[section]
\newtheorem{prop}[lem]{Proposition}
\newtheorem{thm}[lem]{Theorem}

\newtheorem{cor}[lem]{Corollary}

\theoremstyle{definition}

\newtheorem{rem}[lem]{Remark}

\title{On the Northcott property of Dedekind zeta functions}
\author{Xavier G\'en\'ereux and Matilde Lal\'in}
\date{}

\address{Xavier G\'en\'ereux:  D\'epartement de math\'ematiques et de statistique,
                                    Universit\'e de Montr\'eal.
                                    CP 6128, succ. Centre-ville.
                                     Montreal, QC H3C 3J7, Canada}\email{xavier.genereux@umontreal.ca}

\address{Matilde Lal\'in:  D\'epartement de math\'ematiques et de statistique,
                                    Universit\'e de Montr\'eal.
                                    CP 6128, succ. Centre-ville.
                                     Montreal, QC H3C 3J7, Canada}\email{matilde.lalin@umontreal.ca}


\begin{document}

\begin{abstract}
The Northcott property for special values of Dedekind zeta functions and more general motivic $L$-functions  was defined by Pazuki and Pengo. We investigate this property for any complex evaluation of Dedekind zeta functions. The results  are more delicate and subtle than what was proven for the function field case in previous work of Li and the authors, since they include some surprising behavior in the neighborhood of the trivial zeros. The techniques include a mixture of analytic and computer assisted arguments.
\end{abstract}

\maketitle

\section{Introduction}

Recently Pazuki and Pengo \cite{PP} considered a variant of the Northcott property for special values of $L$-functions attached to mixed motives. Usually the Northcott property \cite{Northcott} refers to the fact that a set of algebraic numbers with bounded height and bounded degree must be finite. In the number field case, the problem that Pazuki and Pengo study concerns special values of the Dedekind zeta function. For a field $K$ and $s\in \C$ denote by 
\[\zeta_K^*(s):=\lim_{t\rightarrow s}\frac{\zeta_K(t)}{(t-s)^{\ord_{s}(\zeta_K(t))}},\]
the first nonzero coefficient of the Taylor series for $\zeta_K$ around $s$. 

For a fixed $s=n\in \Z$ and a fixed positive real number $B$, 
Pazuki and Pengo study the set of isomorphism classes of number fields $K$ given by \begin{equation*} S_{B,n}=\{[K]: |\zeta_K^*(n)|\leq B\},\end{equation*}
and discuss the finiteness of this set under various conditions of $B$ and $n$. For number fields, they prove that the Northcott property holds for $n$ a negative integer or $n=0$, and it does not hold if $n$ is a positive integer.  They also estimate the size of this set for the integers $n$ such that the Northcott property holds.

In \cite{GLL} Li and the authors of this note consider the analogous problem for isomorphism classes of function fields $K$ with constant field $\F_q$. However, instead of restricting to special values with $s=n\in \Z$, they work directly with $\zeta_K^*(s)$ with $s$ a fixed arbitrary complex number. They are able to establish or partially establish the question of the Northcott property {\em outside} the set $\frac{1}{2}-\frac{\log 2}{\log q}\leq \re(s)< \frac{1}{2}$. More precisely,  the Northcott property holds when $\re(s)<\frac{1}{2}-\frac{\log 2}{\log q}$ and the set $S_{B,s}$ is infinite for $B$ larger than a certain constant depending on $s$ in 
$\re(s)\geq  \frac{1}{2}$. Moreover,  remark that the gap corresponding to $\frac{1}{2}-\frac{\log 2}{\log q}\leq \re(s)< \frac{1}{2}$ shrinks to the empty set as $q$ tends to infinity. Figure \ref{fig:functionfield} illustrates what is known in the function field case. 
\begin{figure}
    \centering
    \begin{tikzpicture}[scale=2]
        \filldraw [blue!40] (1/3,-2) -- (1/3,2) -- (-2,2) -- (-2,-2);
        \filldraw [red!40] (1,-2) -- (1,2) -- (2,2) -- (2,-2);
        \filldraw [red!40] (1/2,-2) -- (1/2,2) -- (1,2) -- (1,-2);
        \draw [dashed] (1,-2) -- (1,2);
        \draw (-2,0) -- (2,0);
        \draw (0,-2) -- (0,2);
        \draw [dashed] (1/2,-2) -- (1/2,2);
        \node at (1/2,-0.18) {$1/2$};
        \node at (1.1,-0.16) {$1$};
        \node at (0.1,-0.16) {$0$};
    \end{tikzpicture}
    \caption{For the base field $\mathbb{F}_{q}$, with $q$ a fixed prime power, the Northcott property holds in the blue area. 
    In the red area,  $S_{B,s}$ is infinite for $B$ greater than a certain constant (which is zero in the case of the real segment $[1/2,1]$). The white gap disappears when $q \rightarrow \infty$.}
    \label{fig:functionfield}
\end{figure}
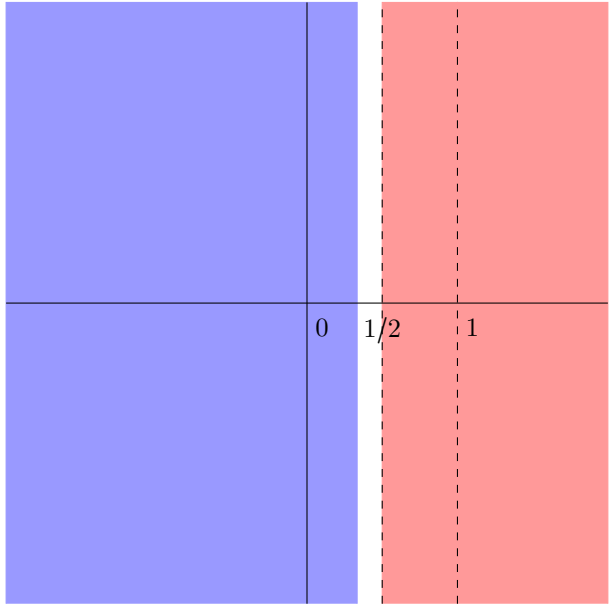
These results are consistent with what Pazuki and Pengo obtained for the cases $s=n\in \Z$.

The goal of this article is to return to the case considered by Pazuki and Pengo in \cite{PP} and to explore $\zeta_K^*(s)$ for the Dedekind zeta function and $s$ an {\em arbitrary} complex number. The motivation for considering such questions comes naturally from a desire to better understand the results in \cite{GLL}. Figure \ref{fig:argh} illustrates the  results obtained in this article over number fields.

\begin{figure}
    \centering
    \begin{tikzpicture}[scale=2]
        \filldraw [blue!40] (-0,-2) -- (-0,2) -- (-4.5,2) -- (-4.5,-2);
        
        \draw[fill=white](-1,0) circle (0.2);
        \draw[fill=white](-2,0) circle (0.16);
        \draw[fill=white](-3,0) circle (0.08);
        \draw[fill=white](-4,0) circle (0.04);

         \draw[fill=red!40](-1,0) circle (0.064);
        \draw[fill=red!40](-2,0) circle (0.032);
        \draw[fill=red!40](-3,0) circle (0.016);
        \draw[fill=red!40](-4,0) circle (0.008);

          \draw[fill=blue!40](-1,0) circle (0.02);

             \draw [black!90, fill=white] plot [smooth,domain=-2:2]  ({-exp(-(1*\x)^2)/2-0.05},{\x});
        
         \filldraw [white] (-0.065,-2) -- (-0.065,2) -- (0,2) -- (0,-2);

        \filldraw [red!40] (1,-2) -- (1,2) -- (2,2) -- (2,-2);
        \filldraw [red!40] (1/2,-2) -- (1/2,2) -- (1,2) -- (1,-2);
        \draw (-4.5,0) -- (2,0);
        \draw (0,-2) -- (0,2);
        \draw [dashed] (1/2,-2) -- (1/2,2);
        \draw [dashed] (1,-2) -- (1,2);
        \node at (1/2,-0.18) {$1/2$};
        \node at (1.1,-0.16) {$1$};
        \node at (0.1,-0.16) {$0$};

    \end{tikzpicture}\caption{Approximate illustration of what is proven in the article regarding the Northcott property for Dedekind zeta functions. The Northcott property holds in the blue area. In the red area, $S_{B,s}$ is infinite for $B$ greater than a certain constant.}
        \label{fig:argh}
\end{figure}
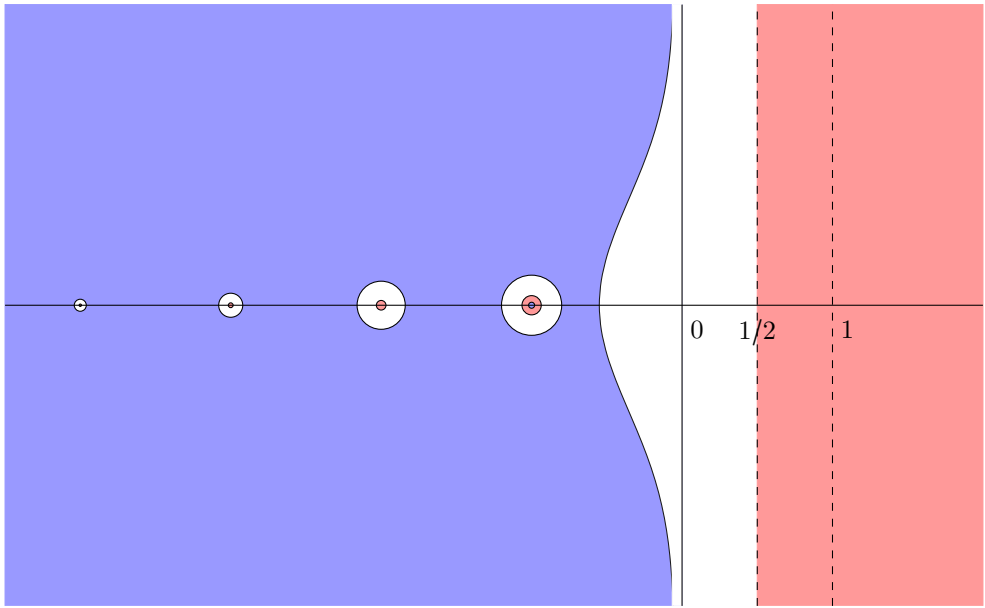

We will say that $s$ satisfies the Northcott property for $B$ a real positive number if $S_{B,s}$ is finite. We prove the following statements.

\begin{thm} \label{thm:Northcottnegativesigma}
The {\em Northcott property is satisfied} for $s=\sigma+i\tau$ with $\sigma, \tau \in \R$ and any $B>0$ under the following conditions: \begin{itemize}
\item When $s=\sigma+i \tau$ with  \[\sigma<-1.5\quad \mbox{  and } \quad \tau>\tau_0:= \frac{2}{\pi}\tanh^{-1}\left(\frac{\zeta\left(\frac{5}{2}\right)}{3\sqrt{2}e^{2\gamma}}\right)=
0.063666\dots,\]
where $\gamma=$ is the Euler--Mascheroni constant given by 
$\gamma=0.577215\dots.$
\item When  $s=\sigma+i\tau = -2n+re^{i\theta}$ with $n \in \Z_{>0}$ and satisfying the following conditions  \[-2n-\frac{1}{2}\leq \sigma \leq -2n+\frac{1}{2}, \quad 
\mbox{ and } \quad |\tau|\leq\tau_0,\] 
as well as  
\begin{align*}
    r> \max\left\lbrace\frac{\sin^{-1}(C_\C(n))}{\pi}, \frac{2\sin^{-1}(C_\R(n))}{\pi}\right\rbrace
\end{align*}
where
\begin{align*}
    C_\C(n)  = \pi\left(\frac{e^{-2\gamma}}{2}\right)^{4n}  \frac{\zeta\left(2n+\frac{1}{2}\right)^2}{\Gamma\left(2n+\frac{1}{2}\right)^2}\frac{18e^{4\gamma} +\zeta\left(\frac{5}{2} \right)^2}{18e^{4\gamma} -\zeta\left(\frac{5}{2} \right)^2},
\end{align*}
and 
\begin{align*}
    C_{\R} (n) =  \sqrt{\pi}
    \left(\frac{e^{-2\gamma}}{2}\right)^{2n}  \frac{\zeta(2n+\frac{1}{2})}{\Gamma(2n+\frac{1}{2})}\frac{3 e^{2\gamma}}{\left(18e^{4\gamma} -\zeta\left(\frac{5}{2} \right)^2\right)^\frac{1}{2}}.
\end{align*}

    \item 

When $s=\sigma+i\tau=-2n+1+re^{i\theta}$ with $n \in \Z_{>1}$ and satisfying the following conditions  \[-2n+\frac{1}{2}\leq \sigma \leq -2n+\frac{3}{2},\quad \sigma<\sigma_0, \quad \mbox{ and }\quad |\tau|\leq\tau_0,\] as well as
\begin{align*}
    r> \frac{1}{\pi}\sin^{-1}\left(\pi \left(\frac{e^{-\gamma}}{2}\right)^{4n-2}\frac{\zeta\left(2n-\frac{1}{2}\right)^2}{\Gamma\left(2n-\frac{1}{2}\right)^2}\frac{18e^{4\gamma}+\zeta\left(\frac{5}{2} \right)^2}{18e^{4\gamma}-\zeta\left(\frac{5}{2} \right)^2}\right).
\end{align*}
\item When $s=\sigma+i \tau$ with $\sigma<0$ and 
\[|\tau| >\frac{\left(2e^{\gamma}\right)^{2\sigma-1}}{\tanh\left(\frac{\pi}{2} \left(2e^{\gamma}\right)^{2\sigma-1}\frac{\zeta(1-\sigma)^2}{\Gamma(1-\sigma)^2}\right)}\frac{\zeta(1-\sigma)^2}{\Gamma(1-\sigma)^2}.\]
\end{itemize}
\end{thm}
The idea behind Theorem \ref{thm:Northcottnegativesigma} is to use the functional equation of the Dedekind zeta function in order to compare the value of $\zeta_K(s)$ with that of $\zeta_K(1-s)$, where $\re(1-s)=1-\sigma>0$, which is easier to understand and control. However, this strategy requires the control of the discriminant $\Delta_K$ and the $\Gamma$-factors. To control the discriminant we use a result of Odlyzko \cite{Odlyzko-lower2} that gives a lower bound for 
the inferior limit of the root discriminant $|\Delta_K|^\frac{1}{[K:\Q]}$ as the degree $[K:\Q]$ goes to infinity. With these bounds in hand, it remains to  control the $\Gamma$-factors, which is done in stages, first when the imaginary part $\tau$ is sufficiently away from zero, and then in discs centered at negative integers, and chosen in such a way that they cover all the strip around the real negative axis, except for some smaller concentric discs. There are several strategies to bound $\tau$ and to choose the discs.  For Theorem \ref{thm:Northcottnegativesigma}, the choice of $\sigma_0$ determines the choice of $\tau_0$, but the choice of $\sigma_0=-1.5$ is arbitrary. It is related with the fact that this method yields very sub-optimal results if we try to reach $-1$, in the sense that the $\tau_0$ must be very large. 
 Given that we do our analysis over intervals of length 1 centered at negative integers, the choice of $\sigma_0=-1.5$ is then natural. Figure \ref{fig:negativesigma} illustrates the strategy and results of Theorem
\ref{thm:Northcottnegativesigma}, except for the last item, which removes the condition $\sigma<\sigma_0$, but it gives a relatively bad bound for $\tau$.

 More precise results   can be obtained by studying the region where the  Northcott property holds with a computer-generated graph (see Figure \ref{fig:pointgraph1}) and by approximating its boundary with analytic methods. A strategy following this idea is described in Section \ref{sec:around0}, and this allows us to numerically prove  that if
$s=\sigma+i\tau=-1+re^{i\theta}$ is such that \[-1.5\leq \sigma \leq \sigma_1,\]
where $\sigma_1\approx -0.68$ is a solution to 
\[\frac{(2e^\gamma)^{\frac{1}{2}-\sigma}}{\zeta(1-\sigma)}\left|\frac{\Gamma(1-\sigma)}{\Gamma(\sigma)}\right|^\frac{1}{2}=1\]
and 
\[r>\num{9.260260274818e-2},\]
then $s$ satisfies the  Northcott property for any $B>0$, 
 which complements the statement of Theorem \ref{thm:Northcottnegativesigma}. In fact, we can be more precise about this, and give information in the interval $-1.5\leq \sigma \leq \varepsilon$ (see Remarks \ref{rem: numproof} and \ref{rem: numproof2} and the discussions following  them). Table \ref{tab: values} exhibits a comparison of the numerical results and the results of Theorem \ref{thm:Northcottnegativesigma}.

\begin{figure}
    \centering
    \begin{tikzpicture}[scale=2]
        \filldraw [blue!30] (-1/4,-2) -- (-1/4,2) -- (-4,2) -- (-4,-2);
        
        \draw[color=blue!70, dashed, very thick](-0.8,0) circle (0.4);
        \draw[color=blue!70, dashed, very thick](-1.4,0) circle (0.4);
        \draw[color=blue!70, dashed, very thick](-2,0) circle (0.4);
        \draw[color=blue!70, dashed, very thick](-2.6,0) circle (0.4);
        \draw[color=blue!70, dashed, very thick](-3.2,0) circle (0.4);
        \draw[color=blue!70, dashed, very thick](-3.8,0) circle (0.4);
        \filldraw [white!20] (-4,-2) -- (-4,2)-- (-4.5,2) -- (-4.5,-2);
        \filldraw [white!20] (0,-2) -- (0,2)-- (-0.5,2) -- (-0.5,-2);
        
        \draw[fill=white](-0.8,0) circle (0.1);
        \draw[fill=white](-1.4,0) circle (0.08);
        \draw[fill=white](-2,0) circle (0.06);
        \draw[fill=white](-2.6,0) circle (0.04);
        \draw[fill=white](-3.2,0) circle (0.02);
        \draw[fill=white](-3.8,0) circle (0.005);

        \draw (-4,0) -- (0,0);
        \draw (0,-2) -- (0,2);
        \draw (-0.5,-0.1) -- (-0.5,0.1);
        \node at (-0.4,-0.16) {$\sigma_0$};
        \node at (0.1,-0.16) {$0$};
    \end{tikzpicture}\caption{Illustration of the Northcott property as verified by the first three items in Theorem \ref{thm:Northcottnegativesigma}. }
        \label{fig:negativesigma}
\end{figure}
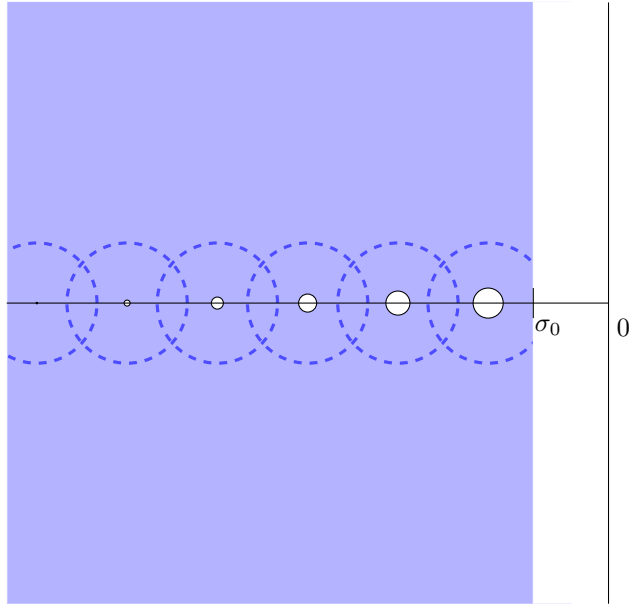

In addition, we provide an estimate for the number of elements in $S_{B,n}$ in the above cases.  
\begin{thm}\label{thm:Couveignesapply}
Let $s= \sigma+i\tau$ such that any of the conditions in Theorem \ref{thm:Northcottnegativesigma} are satisfied.  Then, there are constants $a_s, b_s$ depending only on $s$ such that
\begin{equation}\label{eq:Sbound}\#S_{B,s}\leq \exp \left(b_s(\log B) \left(\log\left(\frac{\log B}{a_s}\right)\right)^3 \right).\end{equation}
\end{thm}
Theorem \ref{thm:Couveignesapply} is obtained by applying a result of Couveignes \cite{Couveignes}
giving a bound for the number of $K$ of fixed degree over $\Q$ and fixed discriminant. This strategy was already employed both in \cite{PP} and in \cite{GLL}, and  Theorem \ref{thm:Couveignesapply} has the same strength as the equivalent results obtained in these articles.

The simplest negative result to examine is the right side of the critical strip. 

\begin{thm} The {\em Northcott property does not hold} for $s=\sigma+i\tau$ with $\sigma, \tau \in \R$, $\sigma>1$ and $B\geq \zeta(\sigma)^2$. 
\end{thm}
This follows from the fact that, for quadratic fields, $\zeta_K(\sigma+i\tau)\leq \zeta(\sigma)^2$. 

We also have negative results on the left side of the critical strip, in the neighborhood of negative integers, but not on the negative integers themselves (where the Northcott property holds as proven in \cite{PP}).

\begin{thm} \label{thm:nonNorthcottnegativesigma}
The {\em Northcott property does not hold} for $s=\sigma+i \tau=-n +re^{i\theta}$ with $n \in \Z_{>0}$, and satisfying the following conditions 
\[-n-\frac{1}{2}\leq \sigma \leq -n+\frac{1}{2}\]
as well as
 \begin{align*}
     0<r < \frac{1}{\pi} \sinh^{-1} \left(  \frac{\pi}{\Gamma\left(n+\frac{3}{2}\right)^2\zeta\left(n+\frac{1}{2}\right)^2}\left(\frac{2\pi}{D_M}\right)^{2n+2} \right),
 \end{align*}
 where 
 \[D_M=3^{\frac{1}{8}}\cdot 7^{\frac{1}{12}}\cdot 13^{\frac{1}{12}}\cdot 19^{\frac{1}{6}}\cdot 23^{\frac{1}{3}}\cdot 29^{\frac{1}{12}} \cdot 31^{\frac{1}{12}}\cdot 35509^{\frac{1}{6}}=78.4269\dots.\]
\end{thm}
Theorem \ref{thm:nonNorthcottnegativesigma} is obtained by applying a result of Hajir, Maire, and Ramakrishna \cite{HajirMaireRamakrishna}, which is an improvement of results of Martinet \cite{Martinet} giving upper bounds for the inferior limit of the root discriminant $|\Delta_K|^\frac{1}{[K:\Q]}$ as the degree $[K:\Q]$ goes to infinity. We remark here that the set where we can prove that the Northcott property is not verified is a punctured disc around each negative integer. 

There is a ring representing a gap of knowledge between Theorems \ref{thm:Northcottnegativesigma} and \ref{thm:nonNorthcottnegativesigma} (see Figure \ref{fig:knowledge gap}). This lack of knowledge originates from the gap between the lower and upper bounds for the inferior limits of the root discriminants of number fields. Also rough bounds for the $\Gamma$-factors contribute to this gap, but these in principle could be improved. 

\begin{figure}
    \centering
    \begin{tikzpicture}[scale=2]
        \filldraw [blue!40] (2,-2) -- (2,2) -- (-2,2) -- (-2,-2);
        \draw[fill=white](0,0) circle (1);
        \draw[fill=red!40](0,0) circle (0.2);
        \draw (-2,0) -- (2,0);
        \draw[fill=blue](0,0) circle (1pt);
        \node at (0.1,-0.08) {$n$};
    \end{tikzpicture}\caption{Illustration of the knowledge gap between Theorems  \ref{thm:Northcottnegativesigma} and \ref{thm:nonNorthcottnegativesigma}. As usual, the Northcott property is verified in the blue area (this includes the center of the circle), while the red area is known to be non-Northcott.}
        \label{fig:knowledge gap}
\end{figure}
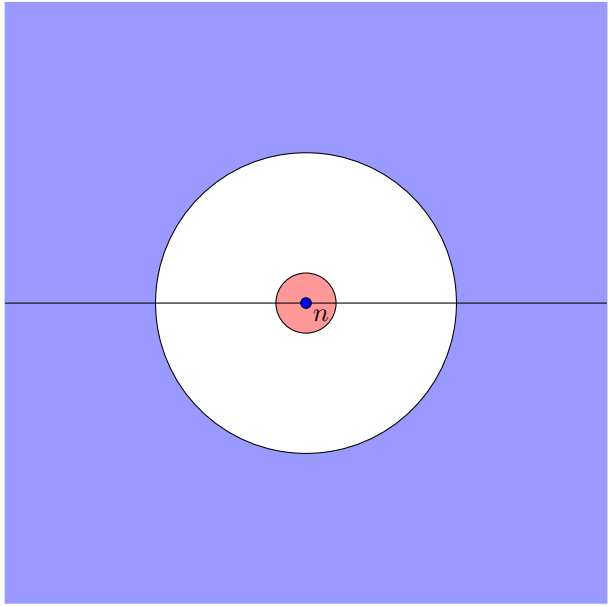

Finally, inside the critical strip, we have the following result.
\begin{thm}\label{thm:critical} Assume the Generalized Riemann Hypothesis. 
Then the {\em Northcott property does not hold} for  $s=\sigma+i\tau$ with  $1/2<\sigma<1$.

Moreover, unconditionally, there is a $B(s)>0$ (given by \eqref{eq:B}) such that $s$ does not satisfy the Northcott  property for $B>B(s)$.
\end{thm}
The first part of Theorem \ref{thm:critical} is obtained by applying a result of Lamzouri \cite{Lamzouri} on the distribution of extreme values in families of quadratic Dirichlet $L$-functions that allows to construct arbitrarily many quadratic extensions with bounded Dirichlet $L$-function at $s$. A result of Sono \cite{Sono} on the second moment of quadratic Dirichlet $L$-functions allows us to prove the second part. 

The main difference between the results for number fields and the analogue results for function fields from \cite{GLL} lies when $\re(s)<0$. While the function field case has a relatively straightforward verification of the Northcott property for $\re(s)<\frac{1}{2}-\frac{\log 2}{\log q}$, this verification fails in a neighborhood of each negative integer in the number field case. This surprising difference comes from the $\Gamma$-factors in the functional equation. 
Another difference lies in the interior of critical strip. More precisely, in the strip where $1/2<\re(s)<1$, we have, conditionally on the Generalized Riemann Hypothesis,  non-Northcott for any $B>0$ in the number field case, as opposed to results that are partial (for $B$ larger than certain value) for $\im(s)\not = 0$ in the function field case. This is due to the strength of the result in \cite{Lamzouri} and can likely be translated to the function field case as well, where the Riemann Hypothesis is known.

This article is organized as follows. Section \ref{sec:background}  includes some necessary background on the Dedekind zeta function, the $\Gamma$ function, and hyperbolic trigonometric functions. The right side of the critical strip is considered in Section \ref{sec:right}.  The left side is treated in Section \ref{sec:left}, while the case $-1.5<\sigma$ is also examined numerically in Section \ref{sec:around0}. Finally, Section \ref{sec:inside} treats the behavior inside the critical strip.  

\section*{Acknowledgements} The authors are grateful to Chantal David, Alexandra Florea, Wanlin Li, and Riccardo Pengo for many helpful discussions. This work was supported by  the Natural Sciences and Engineering Research Council of Canada, Discovery Grant \texttt{355412-2022}, and the Fonds de recherche du Qu\'ebec - Nature et technologies, Projet de recherche en \'equipe \texttt{300951}.

\section{Some background}\label{sec:background}

In this section we recall some background regarding the Dedekind zeta function $\zeta_K(s)$, the Gamma function, and some inequalities involving trigonometric and hyperbolic functions. 
 Let $K$ be a number field, that is, a finite extension of $\Q$ of degree $d_K=r_1+2r_2$, where $r_1$ denotes the number of real embeddings and $r_2$, the number of pairs of complex embeddings. The Dedekind zeta function of $K$ is given by
\begin{equation}\label{eq:Dedekinddef}\zeta_K(s):=\sum_{I\subseteq {\mathcal{O}}_K}\frac{1}{N_{K/\Q}(I)^{s}}=\prod_{P\subseteq {\mathcal{O}}_K} (1-N_{K/\Q}(P)^{-s})^{-1}, \qquad \re(s)>1,\end{equation}
where the sum takes place over all the ideals in the integral domain $\mathcal{O}_K$ and the Euler product goes over the prime ideals of $\mathcal{O}_K$.

Let $\Delta_K$ denote the discriminant of $K/\Q$. The Dedekind zeta function $\zeta_K(s)$ satisfies the following functional equation 
\[\zeta_K(s)=\zeta_K(1-s)\frac{\Gamma_\R(1-s)^{r_1}\Gamma_\C(1-s)^{r_2}}{\Gamma_\R(s)^{r_1}\Gamma_\C(s)^{r_2}} |\Delta_K|^{\frac{1}{2}-s},\]
where
\[\Gamma_\R(s)=\pi^{-s/2}\Gamma(s/2),\qquad \Gamma_\C(s)=2(2\pi)^{-s}\Gamma(s).\]
Here
\[\Gamma(s):=\int_0^\infty x^{s-1} e^{-x} dx, \quad \re(s)>1\]
is the gamma function, which has a meromorphic continuation to the whole complex plane, with simple poles at $0$ and at the negative integers.  When $n$ is a positive integer, we have 
\[\Gamma(n)=(n-1)!\]
The value at a complex number can be controlled by the value at a real argument. More precisely, writing $s=\sigma+i\tau$, we have, 
\begin{equation}\label{eq:boundGamma}
|\Gamma(s)|\leq |\Gamma(\sigma)|.\end{equation}
(See \cite[Eq. 6.1.26]{AS}.)

Euler's reflection formula \cite[Eq. 6.1.17]{AS} gives 
\begin{equation}\label{eq:Euler}\Gamma(s)\Gamma(1-s)=\frac{\pi}{\sin(\pi s)}, \qquad s \not \in \Z.\end{equation}

Lagrange's duplication formula  \cite[Eq. 6.1.18]{AS} yields
\begin{equation}\label{eq:Lagrange} \Gamma(s)\Gamma\left(s+\frac{1}{2}\right)=2^{1-2s}\sqrt{\pi} \Gamma(2s).
\end{equation}

For $\sigma \geq \frac{1}{2}$, we have 
\begin{equation}\label{eq:Gammasinh} 
    |\Gamma(s)| \geq  \frac{\Gamma(\sigma)}{|\cosh(\pi\tau)|^{\frac{1}{2}}}.
\end{equation}
(See \cite[Eq. 5.6.7]{NIST:DLMF}.)

Euler's infinite product (\cite[Eq. 6.1.3]{AS}) gives  
\begin{equation}\label{eq:Gamma-infinite}
\frac{1}{\Gamma(z)}=z e^{\gamma z}  \prod_{k=1}^\infty \left(1+\frac{z}{k}\right)e^{-z/k}
\end{equation}

The digamma function is the logarithmic derivative of the gamma function:
\begin{equation}\label{eq:psi-defn}\psi(s)=\frac{d}{ds} \log(\Gamma(s)).\end{equation}
It can be expressed with the following series (\cite[Eq. 6.3.16]{AS})
\begin{equation} \label{eq:psi-series}
\psi(s+1)=-\gamma +\sum_{k=1}^\infty \left(\frac{1}{k}-\frac{1}{k+s}\right),
\end{equation}
for $s \not = -1,-2,\dots$.

We will need some bounds relating trigonometric functions and hyperbolic functions. 
For example, we have \begin{equation}\label{eq:sinsinh}|\sin(s)| \geq |\sinh(\tau)|.\end{equation}
(See \cite[Eq. 4.3.83]{AS}.)

The following result will be used to bound $|\sin(z)|$ and $|\cos(z)|$ in terms of $|z|$.
\begin{lem} 
\label{sinebound}
For any $z \in \C$, we have 
\begin{align*}
    |\sin(|z|)|\leq |\sin(z)| \leq \sinh(|z|).
\end{align*}
Similarly we have 
\begin{align*}
    |\cos(|z|)|\leq |\cos(z)| \leq \cosh(|z|).
\end{align*}
\end{lem}
\begin{proof} Notice that the upper bound on $|\sin(z)|$ and $|\cos(z)|$ are well-known (see for example \cite[Eq. 4.3.87, Eq. 4.3.86]{AS}).

Write for simplicity $z=re^{i t}$, where $r\geq 0$ and $t \in [0,2\pi)$. Consider
\begin{align*}
    |\sin(z)|^2 = |\sin(re^{i t})|^2 = \sin(r\cos(t))^2\cosh(r\sin(t))^2+\cos(r\cos(t))^2\sinh(r\sin(t))^2=:f(t).
\end{align*}
The derivatives with respect to $t$ give
\begin{align*}
f'(t)=&r\sinh(2r\sin(t))\cos(t)-r\sin(2r\cos(t))\sin(t),\\
f''(t)=&2r^2\cosh(2r\sin(t))\cos(t)^2+2r^2\cos(2r\cos(t))\sin(t)^2\\&-r\sinh(2r\sin(t))\sin(t)-r\sin(2r\cos(t))\cos(t).
\end{align*}
We can focus on $[0,2\pi)$. Studying the derivatives, we find minima at $0$ and $\pi$ as well as maxima at $\frac{1}{2}\pi$ and $\frac{3}{2}\pi$. Thus, we obtain
\[\sin(r)^2\leq f(t)\leq \sinh(r)^2.\]
By taking square roots everywhere in the above inequality we obtain the result.

Similarly, consider 
\begin{align*}
    |\cos(z)|^2 = |\cos(re^{i t})|^2 = \cos(r\cos(t))^2\cosh(r\sin(t))^2+\sin(r\cos(t))^2\sinh(r\sin(t))^2=:g(t).
\end{align*}
As before, we have 
\begin{align*}
g'(t)=&r\sinh(2r\sin(t))\cos(t)+r\sin(2r\cos(t))\sin(t),\\
g''(t)=&2r^2\cosh(2r\sin(t))\cos(t)^2-2r^2\cos(2r\cos(t))\sin(t)^2\\&-r\sinh(2r\sin(t))\sin(t)+r\sin(2r\cos(t))\cos(t). 
\end{align*}
We study the derivatives on $[0,2\pi)$, and we find minima at 0 and $\pi$ and maxima at $\frac{1}{2}\pi$ and $\frac{3}{2}\pi$. This gives 
\[\cos(r)^2\leq g(t)\leq \cosh(r)^2.\]
By taking square roots everywhere in the above inequality we obtain the result.
\end{proof}

\section{The right side of the critical strip} \label{sec:right}
To begin, we consider the Northcott property on the right side of the critical strip, that is, $\C_{\sigma>1}$, where we obtain a  result conditionally on the value of $B$. The result will follow from a comparison between $\zeta_K(s)$ and $\zeta(\sigma)^{d_K}$, where $d_K$ is the degree of the extension $K/\Q$. 
\begin{lem} \label{lem:sandwich}
Let $s=\sigma+i\tau$ with $\sigma>1$. Then $$\frac{1}{\zeta(\sigma)^{d_K}}\leq |\zeta_K(s)| \leq \zeta(\sigma)^{d_K}.$$
\end{lem}
Notice that this result bounds $|\zeta_K(s)|$ by small constants except near $\sigma = 1$. In addition, since $d_K$ can be arbitrary, the above bounds are not absolute for $s$. 

\begin{proof}
We start by proving bounds in terms of $\zeta_K(\sigma)$. Since $\sigma>1$, we have 
\begin{equation}\label{eq:upperboundzeta}|\zeta_K(\sigma+i\tau)|=\left|\sum_{n=1}^\infty \sum_{\substack{I\subseteq \mathcal{O}_K\\N_{K/\Q}(I)=n}}\frac{1}{n^{\sigma+i\tau}}\right|\leq \sum_{n=1}^\infty \left(\ \sum_{\substack{I\subseteq \mathcal{O}_K\\N_{K/\Q}(I)=n}}1\right)\frac{1}{n^{\sigma}}=\zeta_K(\sigma), \end{equation}
and this yields an upper bound. 

To get a lower bound, we take the logarithm of the Euler product \eqref{eq:Dedekinddef} and use the fact that $1+\cos(\theta)\geq 0$ for any $\theta$ to get
\begin{equation}\label{eq:lowerboundzeta}
\log \zeta_K(\sigma)+\re \log \zeta_K(\sigma+i \tau) =\sum_P \sum_{j=1}^\infty \frac{1+\cos (\tau \log |N_{K/\Q}(P)^j|)}{j |N_{K/\Q}(P)^j|^\sigma}\geq 0.
\end{equation}
By combining \eqref{eq:upperboundzeta} and the exponential of \eqref{eq:lowerboundzeta}, we conclude that 
\begin{equation}\label{eq:sandwichzeta}\frac{1}{\zeta_K(\sigma)}\leq |\zeta_K(\sigma+i \tau)|\leq \zeta_K(\sigma).\end{equation}
Finally, using the fact that $N_{K/\Q}(P)$ is a power of the prime $p\in \Z$ lying under the prime ideal $P\subseteq \mathcal{O}_K$, 
we have that 
\[\zeta_K(\sigma)=\prod_{P\subseteq \mathcal{O}_K} (1-N_{K/\Q}(P)^{-\sigma})^{-1}\leq \prod_p \left(1-p^{-\sigma}\right)^{-d_K}=\zeta(\sigma)^{d_K},\]
since there are at most $d_K$ prime ideals  $P$ lying over each $p$. By combining with \eqref{eq:sandwichzeta}, we get the desired result. 
\end{proof}

Combining the above, we arrive at the following result. 
\begin{thm} 
Let $s=\sigma+i\tau$ with $\sigma>1$.
Then the Northcott property does not hold at $s$ for any $B\geq \zeta(\sigma)^2$.

\end{thm}

\begin{proof}
Fix $B \geq \zeta(\sigma)^2$. The upper bound in  Lemma \ref{lem:sandwich} implies  that for any quadratic field $K$,
\begin{align*}
    |\zeta_K(s)| = |\zeta_K(\sigma+i\tau)| \leq \zeta(\sigma)^2.
\end{align*}
This gives  an infinite family of fields with $|\zeta_K(s)| \leq B$ and the result follows.
\end{proof}

\section{The left side of the critical strip} 
\label{sec:left}

We now turn our attention to the left side of the critical strip, namely, $\C_{\sigma<0}$. In this set, Pazuki and Pengo \cite{PP} proved that the Northcott property is satisfied at the negative integers and zero. We will extend this result to show that the Northcott property is satisfied away from the negative integers. We will then see that the Northcott property is not satisfied in a neighborhood around each negative integer that excludes the integer itself.  

Before proceeding to these considerations, we recall some results giving bounds to discriminants in terms of degrees, and prove some basic lemmas. We start by recalling the following statement. 

\begin{thm}[\cite{Odlyzko-lower2},\cite{HajirMaireRamakrishna}]
\label{Odlyzko}
Consider
\[
    \delta(n) = \min_{d_K=n} |\Delta_K|,
\]
that is, the minimum of the absolute values of the discriminants of all the numbers fields of fixed degree $n$ over $\Q$. Let 
\[    D = \liminf_{n\to\infty} \delta(n)^{1/n}.\]

Then we have
\begin{equation*}
    D_m \leq D \leq D_M,
\end{equation*}
where
\begin{align*}
D_m:=&4 \pi e^\gamma=22.3816\dots,\\
D_M:=&3^{\frac{1}{8}}\cdot 7^{\frac{1}{12}}\cdot 13^{\frac{1}{12}}\cdot 19^{\frac{1}{6}}\cdot 23^{\frac{1}{3}}\cdot 29^{\frac{1}{12}} \cdot 31^{\frac{1}{12}}\cdot 35509^{\frac{1}{6}}=78.4269\dots,
\end{align*}
and $\gamma$ is the Euler--Mascheroni constant given by 
\[\gamma=\lim_{n \rightarrow \infty} \left(-\log n +\sum_{k=1}^n \frac{1}{k} \right)=0.57721\dots.\]

\end{thm}
\begin{rem}The lower bound appeared as a culmination of a series of articles by Odlyzko \cites{Odlyzko, Odlyzko-lower, Odlyzko-lower2}. (See also the surveys of Poitou \cite{Poitou} and Odlyzko \cite{Odlyzko-survey}.) Odlyzko's method is a refinement, using ideas of Serre \cite{Serre-vol3}, of an analytic method of Stark \cite{Stark}. This method was a substantial improvement over previous ideas coming directly from Minkowski's bounds. A better lower bound, $8\pi e^\gamma$, is known under the Generalized Riemann Hypothesis. 

The upper bounds come from constructing infinite towers of fields with controlled root discriminant. This idea was due to Martinet \cite{Martinet}, and for a long time, his bound of $2^{3/2}\cdot 11^{4/5}\cdot 23^\frac{1}{2}=93.38\dots$ was the best known. It was later improved by Hajir and Maire \cite{HajirMaire} and finally by  Hajir, Maire, and Ramakrishna \cite{HajirMaireRamakrishna}. The bound in question is obtained by constructing an infinite tower of fields over the totally imaginary field $k=\Q(\alpha)$ with $\alpha$ a root of
$ x^{12}+ 339x^{10}- 19752x^8-2188735x^6+284236829x^4+4401349506x^2+15622982921$, that has degree $d_k=12$ and discriminant 
$|\Delta_k|=7\cdot 13\cdot 19^2\cdot 23^4\cdot 29 \cdot 31\cdot 35509^2$.

\end{rem}

We begin with a simple lemma that transforms the bounds under consideration from field dependent to degree dependent. Together with Lemma \ref{lem:sandwich}, the following statement gives bounds for each of the three factors involved in the functional equation of $\zeta_K(s)$.

\begin{lem} \label{lem:boundsgamma}Let $s\in \C$ and
\begin{align*}
		\Gamma_m (s) &= \min \left\lbrace \left|\frac{\Gamma_\R(1-s)}{\Gamma_\R(s)}\right|, \sqrt{\left|\frac{\Gamma_\C(1-s)}{\Gamma_\C(s)}\right|} \right\rbrace.
\end{align*}		    
    We have the following bound  \begin{align*}
        \Gamma_m(s)^{d_K} \leq \left|\frac{\Gamma_\R(1-s)^{r_1}\Gamma_\C(1-s)^{r_2}}{\Gamma_\R(s)^{r_1}\Gamma_\C(s)^{r_2}}\right|. 
    \end{align*}
\end{lem}
We remark that the above bound is only dependent on the degree of the extension $K/\Q$ and the number of complex and real embeddings, but it is  independent of the field itself.

\begin{proof}
To show this, we consider two  cases. Writing $\gamma_\R(s) = \left|\frac{\Gamma_\R(1-s)}{\Gamma_\R(s)}\right|$ and $\gamma_\C(s) = \left|\frac{\Gamma_\C(1-s)}{\Gamma_\C(s)}\right|^\frac{1}{2}$, we first suppose that $\gamma_\R(s)\leq \gamma_\C(s)$ so that $\Gamma_m(s) = \gamma_{\R}(s)$. Then
\[
    |\gamma_\R^{r_1}(s)\gamma_\C^{2r_2}(s)| \geq |\gamma_\R^{r_1}(s)\gamma_\R^{2 r_2}(s)| = |\Gamma_m(s)|^{d_K}.
\]
Similarly, if $\gamma_\C(s)\leq \gamma_\R(s) $, we  find
\[
    |\gamma_\R^{r_1}(s)\gamma_\C^{2r_2}(s)| \geq |\gamma_\C^{r_1}(s)\gamma_\C^{2r_2}(s)| = |\Gamma_m(s)|^{d_K}.
\]
\end{proof}

The combination of Lemmas \ref{lem:sandwich} and \ref{lem:boundsgamma} yields the following key result, which gives a sufficient condition for having the Northcott property on the left side of the critical strip. 
\begin{prop}
\label{suffnorthcott}
Let $s=\sigma+i\tau$ with $\sigma<0$, and suppose that we have 
\begin{align}\label{condition}
    \frac{\Gamma_m(s)}{\zeta(1-\sigma)}D_m^{\frac{1}{2}-\sigma} > 1.
\end{align}
Then the Northcott property holds at $s$ for any $B>0$. 
\end{prop}

\begin{proof} 
First, we notice that, by Theorem \ref{Odlyzko} there are only finitely many number fields such that
\begin{align*}
    \Delta_K \leq  (D_m-\varepsilon)^{d_K},
\end{align*}
where $\varepsilon>0$ is arbitrary. 

Therefore, for all but finitely many fields $K$, we have that
\begin{align}\label{eq:boundzeta}
    |\zeta_K(s)|\geq |\zeta_K(1-s)| \Gamma_m(s)^{d_K} |\Delta_K|^{\frac{1}{2}-\sigma} \geq \frac{|\Gamma_m(s)|^{d_K}}{\zeta(1-\sigma)^{d_K}} (D_m-\varepsilon)^{d_K(\frac{1}{2}-\sigma)},
\end{align}
where we have applied Lemmas \ref{lem:sandwich} and \ref{lem:boundsgamma}. Choosing $\varepsilon$ sufficiently small, we 
can replace $D_m-\varepsilon$ by $D_m$ and obtain condition \eqref{condition}.

The statement follows from the fact that for large enough degrees, the number on the right hand side of \eqref{eq:boundzeta} is arbitrarily large and that for a fixed degree, there are only finitely many fields with discriminant bounded by any constant.
\end{proof}

A natural question is to bound the size of $S_{B,s}$ in the cases when it is finite. This was done by Pazuki and Pengo \cite{PP} for the case $s=n\in \Z_{n<0}$ using the following   result of Couveignes \cite{Couveignes}.

\begin{thm}\cite{Couveignes}*{Theorem 4, simplified version} \label{thm:Couveignes} There exists an absolute and computable constant $\mathcal{Q}$ such that the following is true. Let $K$ be a number field of degree $n\geq \mathcal{Q}$ and discriminant $\Delta_K$. Then, there are at most 
\[(n^ n|\Delta_K|)^{\mathcal{Q}(\log n)^2}\]
possibilities for $K$.  \end{thm}
 Although Couveignes does not give the value of $\mathcal{Q}$, this constant is only related to the technicalities of the proof  and is independent of the field. 
 
 Using Theorem \ref{thm:Couveignes}, we can prove the following bound, which extends the result of \cite{PP}. 

\begin{thm}
Let $s= \sigma+i\tau$ with $\sigma<0$
and suppose that
\begin{equation}\label{eq:Couveignes-condition} \frac{\Gamma_m(s)}{\zeta(1-\sigma)}D_m^{\frac{1}{2}-\sigma} > 1.\end{equation}
Then, we have
\[\#S_{B,s}\leq \exp \left(2\mathcal{Q}\left(\frac{1}{\frac{1}{2}-\sigma}+\frac{1+\log D_m}{a_s} \right)(\log B) \left(\log\left(\frac{\log B}{a_s}\right)\right)^3 \right),\]
where 
\[a_{s}=\log \left( \frac{\Gamma_m(s)D_m^{\frac{1}{2}-\sigma}}{\zeta(1-\sigma)}\right).\]
\end{thm}

\begin{proof}
By Theorem \ref{thm:Couveignes}, we have that 
\[\#\{[K]\, :\, |\Delta_K|= x, d_K=d  \}\leq d^{\mathcal{Q}d (\log d)^2} x^{\mathcal{Q} (\log d)^2}.\]
Thus, we have 
\begin{align*}
\#\{[K]\, :\, |\Delta_K|\leq X, d_K\leq D \}\leq & \sum_{x=1}^X \sum_{d=1}^D d^{\mathcal{Q}d (\log d)^2} x^{\mathcal{Q} (\log d)^2}\\
\leq &  D^{\mathcal{Q}D(\log D)^2+1}
X^{\mathcal{Q}(\log D)^2+1}\\
\leq& \exp \left(2\mathcal{Q}D (\log D)^3+2\mathcal{Q}(\log X) (\log D)^2 \right).
\end{align*}

Now suppose that 
\[|\zeta_K(s)|\leq B.\]
By equation \eqref{eq:boundzeta}, we must have 
\[|\zeta_K(1-s)| \Gamma_m(s)^{d_K} |\Delta_K|^{\frac{1}{2}-\sigma}\leq B.\]
Combining the above with  \eqref{eq:Couveignes-condition}, we must have
\[|\zeta_K(1-s)| \left(\frac{\zeta(1-\sigma)}{D_m^{\frac{1}{2}-\sigma}}\right)^{d_K} |\Delta_K|^{\frac{1}{2}-\sigma}\leq B.\]
Now we apply Lemma \ref{lem:sandwich} to conclude
\[\frac{|\Delta_K|^{\frac{1}{2}-\sigma}}{D_m^{d_K(\frac{1}{2}-\sigma)}}\leq B.\]
Equation \eqref{eq:boundzeta} also implies that 
\[\left(\frac{|\Gamma_m(s)|D_m^{\frac{1}{2}-\sigma}}{\zeta(1-\sigma)} \right)^{d_K}\leq B.\]

Thus, we obtain, 
\[d_K\leq \frac{\log B}{\log \left( \frac{\Gamma_m(s)D_m^{\frac{1}{2}-\sigma}}{\zeta(1-\sigma)}\right)}\quad 
\mbox{ and }\quad |\Delta_K|\leq B^\frac{1}{\frac{1}{2}-\sigma} B^\frac{\log D_m}{\log \left( \frac{\Gamma_m(s)D_m^{\frac{1}{2}-\sigma}}{\zeta(1-\sigma)}\right)}.\]
Setting 
\[D=\frac{\log B}{a_s},\qquad X= B^{\frac{1}{\frac{1}{2}-\sigma}+\frac{\log D_m}{a_s}},\quad \mbox{ where } \quad a_s = \log \left( \frac{\Gamma_m(s)D_m^{\frac{1}{2}-\sigma}}{\zeta(1-\sigma)}\right),\]
we obtain 
\begin{align*}
\#S_{B,s}\leq &\exp \left(2\mathcal{Q}\frac{\log B}{a_s} \left(\log\left(\frac{\log B}{a_s}\right)\right)^3+2\mathcal{Q}\left(\frac{1}{\frac{1}{2}-\sigma}+\frac{\log D_m}{a_s} \right)(\log B) \left(\log\left(\frac{\log B}{a_s}\right)\right)^2 \right)\\
\leq & \exp \left(2\mathcal{Q}\left(\frac{1}{\frac{1}{2}-\sigma}+\frac{1+\log D_m}{a_s} \right)(\log B) \left(\log\left(\frac{\log B}{a_s}\right)\right)^3 \right).
\end{align*}

\end{proof}

\subsection{Away from the real line}
\label{sec:away negativeintegers}

The next step is to give a general idea of the values of $s$ that respect the condition in Proposition \ref{suffnorthcott}.
To do this, it is useful to search for lower bounds for $\Gamma_m(s)$,  which follow from individual bounds for the ratios $\gamma_\R(s)=\left|\frac{\Gamma_\R(1-s)}{\Gamma_\R(s)}\right|$ and $\gamma_\C(s)=\left|\frac{\Gamma_\C(1-s)}{\Gamma_\C(s)}\right|^\frac{1}{2}$.

\begin{lem}
 \label{lem:lowerbounGamma} Let $s=\sigma +i\tau \in \C\setminus \Z$ with   $\sigma<0$. Then 
\begin{align*}
    \gamma_\C(s)
    &\geq \frac{(2\pi)^{\sigma-\frac{1}{2}}}{\sqrt{\pi}}\Gamma(1-\sigma) \tanh(\pi \tau)^\frac{1}{2}.
\end{align*}

and  
\begin{align*}
   \gamma_\R(s)
    &\geq \sqrt{2}\frac{(2\pi)^{\sigma-\frac{1}{2}}}{\sqrt{\pi}}\Gamma\left(1-\sigma\right) \left|\tanh\left(\frac{\pi\tau}{2}\right) \right|.
\end{align*}
\end{lem}

\begin{proof} First consider the bound for $\Gamma_\C$. By applying Euler's reflection formula \eqref{eq:Euler} as well as inequalities \eqref{eq:Gammasinh} and \eqref{eq:sinsinh}, we obtain 
\begin{align}
    \gamma_\C(s)^2
    &=\left|\frac{2(2\pi)^{s-1}\Gamma(1-s)}{2(2\pi)^{-s}\Gamma(s)}\right|\nonumber \\
    &= (2\pi)^{2\sigma-1}|\Gamma(1-s)|^2 \frac{|\sin(\pi s)|}{\pi}\label{eq:taucomplex}\\
    &\geq \frac{(2\pi)^{2\sigma-1}}{\pi}\Gamma(1-\sigma)^2 \left|\frac{\sin(\pi s)}{\cosh(\pi \tau)}\right|\label{eq:complexsigmabound}\\
    &\geq \frac{(2\pi)^{2\sigma-1}}{\pi}\Gamma(1-\sigma)^2 |\tanh(\pi \tau)|.\nonumber
\end{align}

Now consider the bound for $\Gamma_\R$. Again we apply Euler's reflection formula \eqref{eq:Euler} as well as inequalities \eqref{eq:Gammasinh} and \eqref{eq:sinsinh}, together with Lagrange's duplication formula \eqref{eq:Lagrange}.
\begin{align}
    \gamma_\R(s) &= \pi^{\sigma-\frac{1}{2}}\left|\frac{\Gamma\left(\frac{1-s}{2}\right)}{\Gamma\left(\frac{s}{2}\right)}\right|\nonumber \\
    &= \pi^{\sigma-\frac{1}{2}}\left|\Gamma\left(\frac{1-s}{2}\right)\Gamma\left(1-\frac{s}{2}\right)\right|\frac{\left|\sin(\frac{\pi s}{2})\right|}{\pi}\label{eq:taureal}\\
    &\geq \pi^{\sigma-\frac{1}{2}}\frac{\Gamma\left(\frac{1-\sigma}{2}\right)}{\left|\cosh\left(\frac{\pi\tau}{2}\right)\right|^\frac{1}{2}}\frac{\Gamma\left(1-\frac{\sigma}{2}\right)}{\left|\cosh\left(\frac{\pi\tau}{2}\right)\right|^\frac{1}{2}}\frac{\left|\sin(\frac{\pi s}{2})\right|}{\pi}\nonumber \\
    &{\geq \frac{\pi^{\sigma-\frac{1}{2}}}{\sqrt{\pi}}2^\sigma\Gamma\left(1-\sigma\right)\left|\frac{\sin\left(\frac{\pi s}{2}\right)}{\cosh\left(\frac{\pi\tau}{2}\right)}\right|}\label{eq:realsigmabound} \\
    &\geq
    \sqrt{2}\frac{(2\pi)^{\sigma-\frac{1}{2}}}{\sqrt{\pi}}\Gamma\left(1-\sigma\right) \left|\tanh\left(\frac{\pi\tau}{2}\right) \right|.\nonumber
\end{align}

\end{proof}

The previous results allow us to give a large region to the left of the critical strip where the Northcott property is satisfied. 
  \begin{thm}\label{prop:lowerboundstau0}
Let $s=\sigma+i \tau$ such that  \[\sigma<\sigma_0:=-1.5\quad \mbox{  and } \quad \tau> \tau_0:=\frac{2}{\pi}\tanh^{-1}\left(\frac{\zeta\left(\frac{5}{2}\right)}{3\sqrt{2}e^{2\gamma}}\right)=0.063666\dots.\]
Then the Northcott property holds at $s$ for any $B>0$.
 \end{thm}
 
 \begin{proof}
In fact, for such values of $\sigma, \tau$ we have that 
 \begin{align}\label{eq:and}
\frac{\sqrt{2}\Gamma(1-\sigma) \tanh(\frac{\pi\tau}{2})}{\sqrt{\pi}\zeta(1-\sigma)}\left(\frac{D_m}{2\pi}\right)^{\frac{1}{2}-\sigma} > 1 \quad \mbox{ and }\quad 
\frac{	\Gamma(1-\sigma) \tanh(\pi \tau)^\frac{1}{2}}{\sqrt{\pi}\zeta(1-\sigma)}\left(\frac{D_m}{2\pi}\right)^{\frac{1}{2}-\sigma} > 1.
\end{align}
Remark that the functions $\Gamma(1-\sigma)$, $\left(\frac{D_m}{2\pi}\right)^{\frac{1}{2}-\sigma}=(2e^{\gamma})^{\frac{1}{2}-\sigma}$, and $\frac{1}{\zeta(1-\sigma)}$ all increase as $\sigma$ decreases in the negative part of the real axis.  Now $\tanh(\frac{\pi\tau}{2})$ and $\tanh(\pi \tau)^\frac{1}{2}$ are increasing functions of $\tau$. Thus, the worse possible case is with $\sigma=\sigma_0$ and $\tau=\tau_0$. Fixing $\sigma=-1.5$, we evaluate in $\sigma_0$ and choose $\tau_0$ accordingly so that the inequalities in  \eqref{eq:and} are satisfied.

We then combine the inequalities \eqref{eq:and} with Proposition \ref{suffnorthcott} and Lemma \ref{lem:lowerbounGamma} to conclude. 

 \end{proof}

\subsection{The neighborhood of the real line and away from the integers}
\label{Northcott_negative_integers}
In this section we continue to restrict to the condition $\sigma< \sigma_0$ but we now focus on the case where $|\tau|\leq  \tau_0$. Since $\Gamma_\R$ behaves very differently on odd and even integers, we consider the two cases separately. To cover the remaining strip that has not been covered in Section \ref{sec:away negativeintegers}, we notice that it suffices to investigate the Northcott property in discs centered at negative integers and such that the radii are large enough to cover the whole strip $
|\tau|\leq \tau_0$. In other words, if $s=-m+re^{i\theta}$,  it then suffices to consider  $r\leq\sqrt{\frac{1}{4}+\tau_0^2}$ (see Figure \ref{fig:circle cover}).

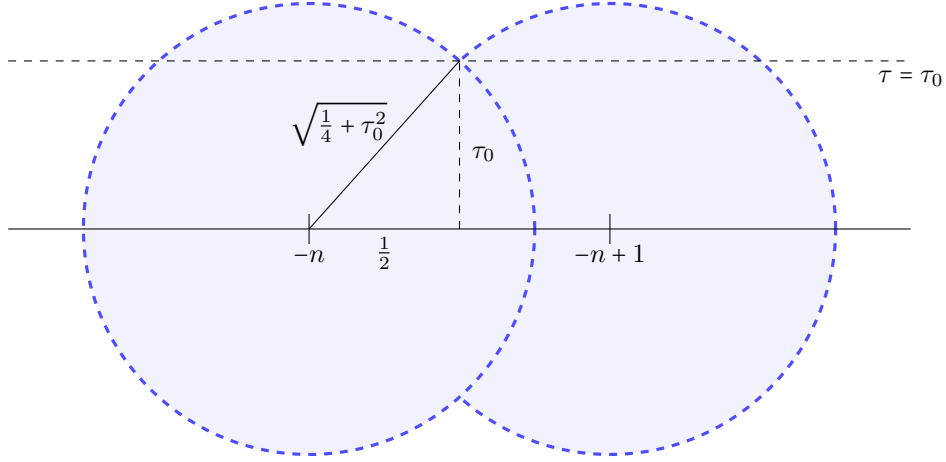
\begin{figure}
    \centering
    \begin{tikzpicture}[scale=2]
        \filldraw[color=blue!70, dashed, fill=blue!5, very thick](1,0) circle (1.5);
        \filldraw[color=blue!70, dashed, fill=blue!5, very thick](-1,0) circle (1.5);
        
        \draw (-3,0) -- (3,0);
        \draw [dashed] (-3,1.118) -- (3,1.118);
        \draw [dashed] (0,0) -- (0,1.118);
        \draw  (-1,0) -- (0,1.118);
        
        \node at (-0.8,0.7) {$\sqrt{\frac{1}{4}+\tau_0^2}$};
        \node at (3,1) {$\tau = \tau_0$};
        \node at (0.16,0.5) {$\tau_0$};
        \node at (-0.5,-0.16) {$\frac{1}{2}$};
        
        \draw (-1,-0.1) -- (-1,0.1);
        \draw (1,-0.1) -- (1,0.1);
        \node at (-1,-0.18) {$-n$};
        \node at (1,-0.16) {$-n+1$};
    \end{tikzpicture}
    \caption{The minimal radius needed to cover the remaining strip $|\tau| \leq \tau_0$ is $\sqrt{\frac{1}{4}+\tau_0^2}$.}
    \label{fig:circle cover}
\end{figure}

The strategy is to determine a specific criterion on $s$ so that it respects the condition in Proposition \ref{suffnorthcott}.

\subsubsection{The negative even integers}
 When we look at even integers, both cases of $\Gamma_m$ are small near $-2n$. However, we will see that when we are sufficiently far from $-2n$, these terms can be compensated by large terms in order to satisfy condition \eqref{condition}. This motivates us to write $s= -2n+re^{i\theta}$, with the goal to get a criterion in terms of $r$. 
We proceed to establish these lower bounds. 
\begin{lem}
 \label{lem:lowerbounGamma2}  Let $s=\sigma+i\tau = -2n+re^{i\theta} \in\C$ be such that it lies in the rectangle  $-2n-\frac{1}{2}\leq \sigma \leq -2n+\frac{1}{2}$ and $|\tau|\leq\tau_0$. Then 
\begin{align*}
    \gamma_\C(s)
    &\geq
    \frac{(2\pi)^{\sigma-\frac{1}{2}}}{\sqrt{\pi}}\Gamma(1-\sigma)\left|\frac{\sin(\pi r)}{\cosh(\pi\tau_0)}\right|^\frac{1}{2}
\end{align*}
and
\begin{align*}
   \gamma_\R(s)
    &\geq \frac{\sqrt{2}}{{\sqrt{\pi}}} (2\pi)^{\sigma-\frac{1}{2}}\Gamma\left(1-\sigma\right)\left|\frac{\sin\left(\frac{\pi r}{2}\right)}{\cosh\left(\frac{\pi\tau_0}{2}\right)}\right|. 
\end{align*}
\end{lem}
\begin{rem}\label{rem:simplification}
Because we have a precise formula for $\tanh\left(\frac{\pi\tau_0}{2}\right)$ from Theorem \ref{prop:lowerboundstau0}, identities such as $\cosh(2\alpha)=\frac{1+\tanh^2(\alpha)}{1-\tanh^2(\alpha)}=2\cosh^2(\alpha)-1$ lead to
\[\cosh(\pi \tau_0)=\frac{18e^{4\gamma} +\zeta\left(\frac{5}{2} \right)^2}{18e^{4\gamma} -\zeta\left(\frac{5}{2} \right)^2}\]
and
\[\cosh\left(\frac{\pi \tau_0}{2}\right)=\frac{3\sqrt{2} e^{2\gamma}}{\left(18e^{4\gamma} -\zeta\left(\frac{5}{2} \right)^2\right)^\frac{1}{2}}.\]
\end{rem}

\begin{proof} First we consider $\Gamma_\C$. By inequality \eqref{eq:complexsigmabound}, Lemma \ref{sinebound}, and the increasing property of $\cosh(x)$ for $x>0$, we have
\begin{align*}
    \gamma_\C(s)
    &\geq \frac{(2\pi)^{\sigma-\frac{1}{2}}}{\sqrt{\pi}}\Gamma(1-\sigma)\left|\frac{\sin(\pi s)}{\cosh(\pi\tau)}\right|^\frac{1}{2} \\
    &\geq \frac{(2\pi)^{\sigma-\frac{1}{2}}}{\sqrt{\pi}}\Gamma(1-\sigma)\left|\frac{\sin(\pi r)}{\cosh(\pi\tau_0)}\right|^\frac{1}{2}.
\end{align*}

Now we consider $\Gamma_\R$. By inequality \eqref{eq:realsigmabound},  we have 
\begin{align*}
    \gamma_\R(s)&\geq \frac{\sqrt{2}}{{\sqrt{\pi}}} (2\pi)^{\sigma-\frac{1}{2}}\Gamma\left(1-\sigma\right)\left|\frac{\sin\left(\frac{\pi s}{2}\right)}{\cosh\left(\frac{\pi\tau}{2}\right)}\right|\\
    &\geq \frac{\sqrt{2}}{{\sqrt{\pi}}} (2\pi)^{\sigma-\frac{1}{2}}\Gamma\left(1-\sigma\right)\left|\frac{\sin\left(\frac{\pi r}{2}\right)}{\cosh\left(\frac{\pi\tau_0}{2}\right)}\right|.
\end{align*}
where we have used again Lemma \ref{sinebound} and the fact that $\cosh(x)$ is an increasing function for $x>0$. 
\end{proof}
We have what we need to show the following result. 
\begin{thm}
\label{propevencircles}
Let $s=\sigma+i\tau = -2n+re^{i\theta} \in\C$ be such that it lies in the rectangle  $-2n-\frac{1}{2}\leq \sigma \leq -2n+\frac{1}{2}$ and $|\tau|\leq\tau_0$. 
Define 
\begin{align*}
    \rho(-2n) = \max\left\lbrace\frac{\sin^{-1}(C_\C(n))}{\pi}, \frac{2\sin^{-1}(C_\R(n))}{\pi}\right\rbrace
\end{align*}
where
\begin{align*}
    C_\C(n)  = \pi\left(\frac{2\pi}{D_m}\right)^{4n}  \frac{\zeta\left(2n+\frac{1}{2}\right)^2}{\Gamma\left(2n+\frac{1}{2}\right)^2}|\cosh(\pi\tau_0)|
\end{align*}
and 
\begin{align*}
    C_{\R} (n) =  \frac{\sqrt{\pi}}{\sqrt{2}}\left(\frac{2\pi}{D_m}\right)^{2n}  \frac{\zeta(2n+\frac{1}{2})}{\Gamma(2n+\frac{1}{2})}\left|\cosh\left(\frac{\pi\tau_0}{2}\right)\right|.
\end{align*}
Then if $r>\rho(-2n)$, the Northcott property holds at $s$ for any $B>0$. 
\end{thm}
We refer to Table \ref{tab: values} for the values of $\rho(-2n)$ at small positive integers $n$.

\begin{proof}
The result follows by combining the previous statements. More specifically, 
Proposition \ref{suffnorthcott} together with Lemma \ref{lem:lowerbounGamma2} give certain criteria for the Northcott property  that apply under the conditions of the statement. Here we  take into account that $\sigma \in \left[-2n-\frac{1}{2},-2n+\frac{1}{2}\right]$ and the strategy will be to find  bounds in terms of $n$ by taking the worst possible cases.

First we consider the complex case. It follows that a sufficient condition for $s$ to satisfy the Northcott property when $\Gamma_m(s)=\gamma_\C(s)$ is
\begin{align*}
   \frac{1}{\sqrt\pi}\left(\frac{D_m}{2\pi}\right)^{\frac{1}{2}-\sigma}  \frac{\Gamma\left(1-\sigma\right)}{\zeta\left(1-\sigma\right)}  \left|\frac{\sin(\pi r)}{\cosh(\pi\tau_0)}\right|^\frac{1}{2}> 1.
\end{align*}
By considering the worst case in each factor, we find
\begin{align*}
   \frac{1}{\sqrt\pi}\left(\frac{D_m}{2\pi}\right)^{2n}  \frac{\Gamma\left(2n+\frac{1}{2}\right)}{\zeta\left(2n+\frac{1}{2}\right)}  \left|\frac{\sin(\pi r)}{\cosh(\pi\tau_0)}\right|^\frac{1}{2}> 1,
\end{align*}
which is the same as requiring
\begin{align} \label{eq:CC}
    C_{\C} (n) =  \pi\left(\frac{2\pi}{D_m}\right)^{4n}  \frac{\zeta\left(2n+\frac{1}{2}\right)^2}{\Gamma\left(2n+\frac{1}{2}\right)^2}|\cosh(\pi\tau_0)| < \sin(\pi r).
\end{align}

In the real case, a sufficient condition  for $s$ to satisfy the Northcott property when $\Gamma_m(s) =\gamma_\R(s)$ is
\begin{align*}
   \frac{\sqrt{2}}{\sqrt{\pi}}\left(\frac{D_m}{2\pi}\right)^{\frac{1}{2}-\sigma}  \frac{\Gamma(1-\sigma)}{\zeta(1-\sigma)} \left|\frac{\sin\left(\frac{\pi r}{2}\right)}{\cosh\left(\frac{\pi\tau_0}{2}\right)}\right| > 1.
\end{align*}
Considering the worst case in each factor gives 
\begin{align*}
   \frac{\sqrt{2}}{\sqrt{\pi}}\left(\frac{D_m}{2\pi}\right)^{2n}  \frac{\Gamma(2n+\frac{1}{2})}{\zeta(2n+\frac{1}{2})} \left|\frac{\sin\left(\frac{\pi r}{2}\right)}{\cosh\left(\frac{\pi\tau_0}{2}\right)}\right| > 1,
\end{align*}
which is the same as requiring
\begin{align}\label{eq:CR}
   C_{\R} (n) = \frac{\sqrt{\pi}}{\sqrt{2}}\left(\frac{2\pi}{D_m}\right)^{2n}  \frac{\zeta(2n+\frac{1}{2})}{\Gamma(2n+\frac{1}{2})}\left|\cosh\left(\frac{\pi\tau_0}{2}\right)\right| <  \sin\left(\frac{\pi r}{2}\right).
\end{align}

With inequalities \eqref{eq:CC} and \eqref{eq:CR} in place, we immediately get the following constraints for $r$
\begin{enumerate}
    \item[(a)] $$\frac{\sin^{-1}(C_\C(n))}{\pi} < r < 1-\frac{\sin^{-1}(C_\C(n))}{\pi},$$
    \item[(b)] $$\frac{2\sin^{-1}(C_\R(n))}{\pi} < r < 2-\frac{2\sin^{-1}(C_\R(n))}{\pi}.$$
\end{enumerate}

Note that as long as 
\begin{equation}\label{eq:conditionrho}
\sqrt{\frac{1}{4}+\tau_0^2}< \min \left\lbrace 1-\frac{\sin^{-1}(C_\C(n))}{\pi}, 2-\frac{2\sin^{-1}(C_\R(n))}{\pi}\right\rbrace,
\end{equation}
 the upper conditions in (a) and (b) are automatically verified. 
The minimum from the right hand side of \eqref{eq:conditionrho} is achieved at $n=1$, namely,
\[ 1-\frac{1}{\pi}\sin^{-1} \left(  \frac{\zeta\left(\frac{5}{2}\right)^2e^{-4\gamma} }{9}|\cosh(\pi\tau_0)|\right)=1-\frac{1}{\pi}\sin^{-1} \left( \frac{\zeta\left(\frac{5}{2}\right)^2\left(18+e^{-4\gamma}\zeta\left(\frac{5}{2}\right)^2\right)}{9\left(18e^{4\gamma}-\zeta\left(\frac{5}{2}\right)^2\right)}\right)=0.993547\dots,\]
but we have  \[\sqrt{\frac{1}{4}+\tau_0^2}=0.504037\dots,\]
and therefore \eqref{eq:conditionrho} is satisfied.

 In conclusion, it suffices to define the  function $\rho$ by 
 \begin{align*}
     \max\left\lbrace\frac{\sin^{-1}(C_\C(n))}{\pi}, \frac{2\sin^{-1}(C_\R(n))}{\pi}\right\rbrace.
 \end{align*}
\end{proof}

Taking the worst case we have the following consequence. 
\begin{cor}
With the same notation  as in Theorem \ref{propevencircles}, the Northcott property holds for $-2n-\frac{1}{2}\leq \sigma \leq -2n+\frac{1}{2}$  and\[r>\rho(-2) =\frac{2}{\pi} \sin^{-1} \left(\frac{\zeta\left(\frac{5}{2}\right)}{\left(18e^{4\gamma}-\zeta\left(\frac{5}{2}\right)\right)^\frac{1}{2}} \right)= 0.063889\dots\] and any $B>0$. 
\end{cor}

\subsubsection{The negative odd integers}

In this case the lower bound for the term involving $\Gamma_\C$ will be similar to what we had in Lemma \ref{lem:lowerbounGamma2}, while the lower bound for the term involving $\Gamma_\R$ will be different. 
\begin{lem}
 \label{lem:lowerbounGamma2odd}
 Let $s=\sigma+i\tau = -2n+1+re^{i\theta} \in\C$ be such that it lies in the rectangle  $-2n+\frac{1}{2}\leq \sigma \leq -2n+\frac{3}{2}$ and $|\tau|\leq\tau_0$. Then 
\begin{align}\label{eq:gammaCbound}
    \gamma_\C(s)
    &\geq \frac{(2\pi)^{\sigma-\frac{1}{2}}}{\sqrt{\pi}}\Gamma(1-\sigma)\left|\frac{\sin(\pi r)}{\cosh(\pi\tau_0)}\right|^\frac{1}{2}
\end{align}
and
\begin{equation} \label{eq:boundgammaRodd}
    \gamma_\R(s)
    \geq  \frac{\sqrt{2}}{\sqrt{\pi}}(2\pi)^{\sigma-\frac{1}{2}}\Gamma\left(1-\sigma\right)
    \left|\frac{\cos(\frac{\pi }{2}\sqrt{\frac{1}{4}+\tau_0^2})}{\cosh\left(\frac{\pi\tau_0}{2}\right)}\right|. 
\end{equation}
\end{lem}
\begin{rem}
As in Remark \ref{rem:simplification}, the precise formula for $\tanh\left(\frac{\pi\tau_0}{2}\right)$ from Theorem \ref{prop:lowerboundstau0} gives us 
\[\cos\left(\frac{\pi }{2}\sqrt{\frac{1}{4}+\tau_0^2}\right)=\cos\left(\frac{1}{2}\sqrt{
\frac{\pi^2}{4}+\left(\log \left(\frac{3\sqrt{2}e^{2\gamma}+\zeta\left(\frac{5}{2} \right) }{3\sqrt{2}e^{2\gamma}-\zeta\left(\frac{5}{2}\right)}\right)\right)^2}\right).\]
\end{rem}
\begin{proof}
The inequality \eqref{eq:gammaCbound} is obtained by following the same steps as in the proof of Lemma \ref{lem:lowerbounGamma2}. Remark that attempting to follow these ideas with $\Gamma_\R$ produces
\begin{align*}
    \gamma_\R(s)
    \geq \frac{\sqrt{2}}{\sqrt{\pi}}(2\pi)^{\sigma-\frac{1}{2}}\Gamma\left(1-\sigma\right)\left|\frac{\sin(\pi\frac{-2n+1+re^{i\theta}}{2})}{\cosh\left(\frac{\pi\tau_0}{2}\right)}\right|,
\end{align*}
which is not small near odd integers. Actually we have 
\begin{align*}
    \left|\sin(\pi\frac{-2n+1+re^{i\theta}}{2})\right| = \left|\cos(\pi\frac{re^{i\theta}}{2})\right| \geq \cos(\frac{\pi r}{2}),
\end{align*}
where the last inequality follows from Lemma \ref{sinebound}.

Since the cosine function $\cos(\frac{\pi r}{2})$ decreases from 0 to 1, it will be the smallest at $r= \sqrt{\frac{1}{4}+\tau_0^2}<1$. We find
\begin{align*}
   \gamma_\R(s)
    \geq \frac{\sqrt{2}}{\sqrt{\pi}}(2\pi)^{\sigma-\frac{1}{2}}\Gamma\left(1-\sigma\right)
    \left|\frac{\cos(\frac{\pi }{2}\sqrt{\frac{1}{4}+\tau_0^2})}{\cosh\left(\frac{\pi\tau_0}{2}\right)}\right|,
    \end{align*}
    which concludes the proof. 

\end{proof}

Since the lower bound \eqref{eq:boundgammaRodd} does not tend to 0 as $r\to 0$, the fact that the Northcott property holds near odd integers will only depend on $n$ and $\tau_0$. In fact
\begin{prop}\label{lem:oddreal}
 Let $s=\sigma+i\tau=-2n+1+re^{i\theta} \in\C$ be such that it lies in the rectangle  $-2n+\frac{1}{2}\leq \sigma \leq -2n+\frac{3}{2}$ and $|\tau|\leq\tau_0$. Suppose that $\Gamma_m(s) = \gamma_\R(s)$. Then if 
\begin{align*}
    \frac{\sqrt{2}}{\sqrt{\pi}}\frac{\Gamma\left(2n-\frac{1}{2}\right)}{\zeta\left(2n-\frac{1}{2}\right)}
    \left|\frac{\cos\left(\frac{\pi }{2}\sqrt{\frac{1}{4}+\tau_0^2}\right)}{\cosh\left(\frac{\pi\tau_0}{2}\right)}\right|
     \left(\frac{D_m}{2\pi}\right)^{2n-1} >1,
\end{align*}
the Northcott property holds for  $s$ and any $B>0$. 
\end{prop}

We can see that in the worst case, namely $n=2$, $\tau_0$ must be at least $0.85$ for this condition to fail. 
\begin{prop}\label{prop:oddcomplex}
 Let $s=\sigma+i\tau=-2n+1+re^{i\theta}  \in\C$ with $\sigma<\sigma_0$ be such that it lies in the rectangle  $-2n+\frac{1}{2}\leq \sigma \leq -2n+\frac{3}{2}$ and $|\tau|\leq\tau_0$. Suppose that $\Gamma_m(s) = \gamma_\C(s)$. Define  
 \begin{align*}
    \rho(-2n+1) = \frac{1}{\pi}\sin^{-1}\left(\pi \left(\frac{2\pi}{D_m}\right)^{4n-2}\frac{\zeta\left(2n-\frac{1}{2}\right)^2}{\Gamma\left(2n-\frac{1}{2}\right)^2}|\cosh(\pi\tau_0)|\right).
\end{align*}
Then if  $r>\rho(-2n+1)$, then the Northcott property holds at $s$ for any $B>0$.
\end{prop}
Refer to Table \ref{tab: values} for the values of $\rho(-2n+1)$ at small positive integers $n$.

\begin{proof}
The proof of this statement follows that sames lines as the proof of Theorem \ref{propevencircles}, although it suffices to  only consider $\Gamma_\C$ this time. Combining Proposition \ref{suffnorthcott} with Lemma \ref{lem:lowerbounGamma2odd} we find
\begin{align*}
   \frac{1}{\sqrt{\pi}} \left(\frac{D_m}{2\pi}\right)^{2n-1}\frac{\Gamma\left(2n-\frac{1}{2}\right)}{\zeta\left(2n-\frac{1}{2}\right)}\left|\frac{\sin(\pi r)}{\cosh(\pi\tau_0)}\right|^\frac{1}{2} > 1,
\end{align*}
giving the condition\begin{align*}
    \sin(\pi r) > \pi \left(\frac{2\pi}{D_m}\right)^{4n-2}\frac{\zeta\left(2n-\frac{1}{2}\right)^2}{\Gamma\left(2n-\frac{1}{2}\right)^2}|\cosh(\pi\tau_0)|.
\end{align*}
We reach the result with the same argument used at the end of the proof of Theorem \ref{propevencircles}.
\end{proof}

\begin{rem}
Propositions \ref{lem:oddreal} and \ref{prop:oddcomplex} together form the third item in Theorem \ref{thm:nonNorthcottnegativesigma}.
\end{rem}

\subsection{The neighborhood of the negative integers}
The goal of this section is to prove non-Northcott near odd negative integers. To do this, we use the constant $D_M=3^{\frac{1}{8}}\cdot 7^{\frac{1}{12}}\cdot 13^{\frac{1}{12}}\cdot 19^{\frac{1}{6}}\cdot 23^{\frac{1}{3}}\cdot 29^{\frac{1}{12}} \cdot 31^{\frac{1}{12}}\cdot 35509^{\frac{1}{6}}=78.4269\dots$ given by Hajir, Maire, and Ramakrishna in \cite{HajirMaireRamakrishna}, and associated to a tower of totally complex fields. More precisely, Hajir, Maire, and Ramakrishna give an infinite sequence of totally complex fields $K_\ell$ satisfying 
\begin{equation}\label{eq:Kell}\lim_{\ell  \rightarrow \infty} |\Delta_{K_\ell }|^{\frac{1}{d_{K_\ell }}}=D_M.\end{equation}

The following result is a natural complement to Proposition \ref{suffnorthcott}. 
\begin{prop}
\label{suffnonNorthcott}
Let $s=\sigma+i \tau$ with $\sigma<0$ and suppose that we have
\begin{align}
    \label{eq: suffnonNorthcott}
    \gamma_\C(s)\zeta(1-\sigma)D_M^{\frac{1}{2}-\sigma} < 1.
\end{align}
Then the Northcott property does not hold at $s$ for any $B>0$. 
\end{prop}
\begin{proof}
Using the sequence of totally complex fields $K_\ell$ satisfying \eqref{eq:Kell}, and fixing an arbitrary $\varepsilon>0$,  we notice that there are infinitely many $K_\ell$ such that
\begin{align}\label{eq:boundKell}
    |\zeta_{K_\ell}(s)|  &= |\zeta_{K_\ell}(1-s)|\left|\frac{\Gamma_\R(1-s)^{r_1}\Gamma_\C(1-s)^{r_2}}{\Gamma_\R(s)^{r_1}\Gamma_\C(s)^{r_2}} \right| \left|\Delta_{K_\ell}^{\frac{1}{2}-s}\right|\nonumber \\
    &\leq \gamma_\C(s)^{d_{K_\ell}}\zeta(1-\sigma)^{d_{K_\ell}}(D_M+\varepsilon)^{d_{K_\ell}(\frac{1}{2}-\sigma)}
\end{align}
Here we have used Lemma \ref{lem:sandwich} together with the fact that the fields $K_\ell$ are totally complex, and all but finitely many must have root discriminant less than $D_M+\varepsilon$.
Now, by setting
$$f_\varepsilon(s) = \gamma_\C(s)\zeta(1-\sigma)(D_M+\varepsilon)^{\frac{1}{2}-\sigma},$$  
 we can write \eqref{eq:boundKell} as
\begin{align*}
    |\zeta_{K_\ell}(s)| < f_\varepsilon(s)^{d_{K_\ell}}.
\end{align*}
When inequality \eqref{eq: suffnonNorthcott} is true, we can also find an $\varepsilon>0$ such that it is still true with $D_M$ replaced by $D_M+\varepsilon$. Therefore we have $f_\varepsilon(s)<1$, and we find that for every $B>0$ there is a sufficiently large $d_B$ such that 
\begin{align*}
    |\zeta_{K_\ell}(s)| < f(s)^{d_{K_\ell}} < B & \text{ for all } d_{K_\ell}> d_B.
\end{align*}

\end{proof}

Since $\Gamma_\C(s)$ has poles on the negative integers, the ratio $\frac{\Gamma_\C(1-s)}{\Gamma_\C(s)}$ vanishes on them and we expect the condition of Proposition \ref{suffnonNorthcott} to hold in small discs around them.
We will use similar ideas to those introduced in Section \ref{Northcott_negative_integers}.

Consider $s=\sigma+i\tau = -n+re^{i\theta} \in\C$, where $n\in \Z_{<0}$ is chosen so that $r$ is minimal. That is to say, we choose $n$ so that $-n$ is the  closest integer to $\sigma$, in other words, $\sigma\in\left[-n-\frac{1}{2},-n+\frac{1}{2}\right]$. 

\begin{thm}
 Let $s=\sigma+i\tau = -n+re^{i\theta} \in\C$ be such that it verifies  $-n-\frac{1}{2}\leq \sigma \leq -n+\frac{1}{2}$. If 
 \begin{align*}
     r < \frac{1}{\pi} \sinh^{-1} \left(  \frac{\pi}{\Gamma\left(n+\frac{3}{2}\right)^2\zeta\left(n+\frac{1}{2}\right)^2}\left(\frac{2\pi}{D_M}\right)^{2n+2} \right),
 \end{align*}
 then the Northcott property does not hold at $s$ for any $B>0$. 
\end{thm}

\begin{proof}
Starting from \eqref{eq:taucomplex}, applying  Lemma \ref{sinebound} and \eqref{eq:boundGamma},  we obtain 
\begin{align*}
   \gamma_\C(s)\zeta(1-\sigma)D_M^{\frac{1}{2}-\sigma} &= |\Gamma(1-s)| \left(\frac{|\sin(\pi s)|}{\pi}\right)^\frac{1}{2}\zeta(1-\sigma)\left(\frac{D_M}{2\pi}\right)^{\frac{1}{2}-\sigma}\\
   &\leq  \Gamma(1-\sigma) \left(\frac{\sinh(\pi r)}{\pi}\right)^\frac{1}{2}\zeta(1-\sigma)\left(\frac{D_M}{2\pi}\right)^{\frac{1}{2}-\sigma}.
\end{align*}

Our goal is to guarantee the condition in Proposition \ref{suffnonNorthcott}. Thus we want,
\begin{align}
    \Gamma(1-\sigma) \left(\frac{\sinh(\pi r)}{\pi}\right)^\frac{1}{2}\zeta(1-\sigma)\left(\frac{D_M}{2\pi}\right)^{\frac{1}{2}-\sigma}<1 \nonumber\\
    \label{eq:boundsigma_to_n}
    \iff \sinh(\pi r)< \frac{\pi}{\Gamma(1-\sigma)^2\zeta(1-\sigma)^2}\left(\frac{2\pi}{D_M}\right)^{1-2\sigma}.
\end{align}
Since $-n -\frac{1}{2}\leq \sigma \leq -n +\frac{1}{2}$, we obtain the result by optimizing each term in 
 \eqref{eq:boundsigma_to_n} under these restrictions. 
\end{proof}

These radii are rather small and decrease quickly. The first few values can be found in the third column in Table \ref{tab: values}. More values are available in \cite{Genereux_The-northcott-property-of-dedekind-zeta-functions_2022}. 
\subsection{The case of $\sigma_0 \leq \sigma \leq 0 $} 
We now turn our attention to the remaining area outside the critical strip, that is, $\sigma_0 \leq \sigma < 0$. Remark that the restriction to $\sigma<\sigma_0$
in the last sections originates from requiring that  $\Gamma(1-\sigma)$  be sufficiently large for condition \eqref{condition} in Proposition \ref{suffnorthcott} to be satisfied when $\tau$ is sufficiently large. As this is no longer the case we turn to other methods.

Direct calculation of condition \eqref{condition} reveals that for fixed $\sigma<0$, we expect
 the Northcott property to hold for $s=\sigma+i\tau$ with $|\tau|>T_\sigma$ for certain $T_\sigma$ depending on $\sigma$. This is best seen in a point graph which verifies \eqref{condition} in a grid. (See 
 Figure \ref{fig:pointgraph1}.)
\begin{figure}
    \centering
    \includegraphics[scale = 0.7]{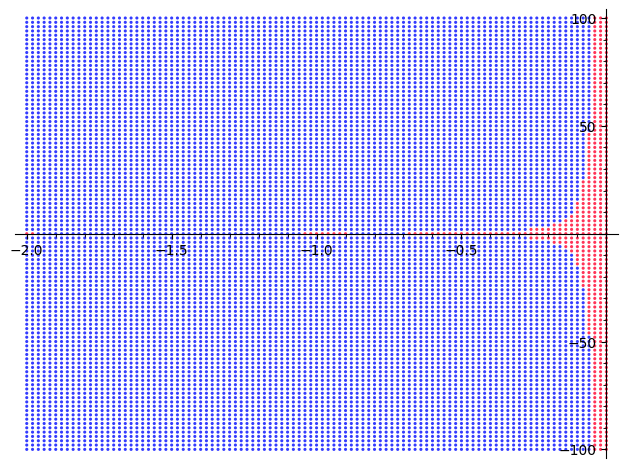}
    \caption{Depiction of points satisfying condition \eqref{condition} corresponding to the Northcott property (in blue).
    } 
    \label{fig:pointgraph1}
\end{figure}
 We can provide an effective result of this statement by slightly modifying the proof of Lemma \ref{lem:lowerbounGamma}. Before doing this, we need the following auxiliary statements.  
\begin{lem} \label{lem:lowergamma} Let $s=\sigma+i\tau$ with $\sigma>1$. Then we have 
\begin{align*}
    |\Gamma(\sigma+i\tau)|^2 \geq |\Gamma(\sigma)|^2\left|\frac{\pi\tau}{\sin(\pi i\tau)}\right|.
\end{align*}
\end{lem}
\begin{proof}
By applying Euler's infinite product \eqref{eq:Gamma-infinite}, and by using that $f(x) = \frac{x^2}{x^2+\tau^2}$ is strictly increasing for $x>0$, we have 
\[\left|\frac{\Gamma(\sigma+i\tau)}{\Gamma(\sigma)}\right|^2= \prod_{k=0}^\infty \frac{\left(k+\sigma\right)^2}{(k+\sigma)^2+\tau^2}\geq  \prod_{k=0}^\infty \frac{(k+1)^2}{(k+1)^2+\tau^2}=\prod_{k=1}^\infty \frac{1}{1+\frac{\tau^2}{k^2}}=\left|\frac{\pi i \tau}{\sin(\pi i\tau)}\right|.\]

\end{proof}

\begin{lem} \label{lem:lowergamma2}
Let $s=\sigma+i\tau$ such that $\sigma<0$. Then we have  
\begin{align}
    \gamma_\C(s)^2 &\geq (2\pi)^{2\sigma-1}\Gamma(1-\sigma)^2|\tau|\label{eq:gammacbound}\\
    \gamma_\R(s)^2 &\geq (2\pi)^{2\sigma-1}\Gamma(1-\sigma)^2|\tau| \left|\tanh\left(\frac{\pi\tau}{2}\right)\right|.\label{eq:gammarbound}
\end{align}
\end{lem}
\begin{proof}
Equation \eqref{eq:gammacbound} follows immediately from equations \eqref{eq:taucomplex}, \eqref{eq:sinsinh}, and Lemma \ref{lem:lowergamma}.

Equation \eqref{eq:gammarbound} follows similarly from \eqref{eq:taureal} and the duplication formula \eqref{eq:Lagrange} since 
\begin{align*}
\gamma_\R(s)=&\pi^{\sigma-\frac{1}{2}}\left|\Gamma\left(\frac{1-s}{2}\right)\Gamma\left(1-\frac{s}{2}\right)\right|\frac{\left|\sin(\frac{\pi s}{2})\right|}{\pi}\\
=& \pi^{\sigma-1}\left|2^s\Gamma\left(1-s\right)\right|\left|\sin(\frac{\pi s}{2})\right|\\
\geq & \pi^{\sigma-\frac{1}{2}}2^\sigma\left|\Gamma\left(1-\sigma\right)\right|\frac{\left|\sin(\frac{\pi s}{2})\right|}{\left|\sinh(\pi\tau)\right|^{\frac{1}{2}}} |\tau|^{\frac{1}{2}}\\
\geq & \pi^{\sigma-\frac{1}{2}}2^\sigma\left|\Gamma\left(1-\sigma\right)\right|\frac{\left|\sinh(\frac{\pi \tau}{2})\right|}{\left|\sinh(\pi\tau)\right|^{\frac{1}{2}}} |\tau|^{\frac{1}{2}}\\
= & (2\pi)^{\sigma-\frac{1}{2}}\left|\Gamma\left(1-\sigma\right)\right|\left|\tanh\left(\frac{\pi\tau}{2}\right)\right|^{\frac{1}{2}} |\tau|^{\frac{1}{2}},
\end{align*}
Where we have used inequality \eqref{eq:sinsinh} and the well-known identity $2\frac{\sinh(\frac{\pi\tau}{2})^2}{|\sinh(\pi\tau)|} = \left|\tanh(\frac{\pi\tau}{2})\right|$.
\end{proof}

Combining the result above with inequality \eqref{condition}, we have the following result. 
\begin{thm} \label{thm: analytical_effective}
Let $s=\sigma+i \tau$ with $\sigma<0$ and 
\[|\tau| >\frac{1}{\tanh\left(\frac{\pi}{2} \left(\frac{D_m}{2\pi}\right)^{2\sigma-1}\frac{\zeta(1-\sigma)^2}{\Gamma(1-\sigma)^2}\right)}\left(\frac{D_m}{2\pi}\right)^{2\sigma-1}\frac{\zeta(1-\sigma)^2}{\Gamma(1-\sigma)^2}.\]
Then, the Northcott property holds at $s$ for any $B>0$. 
\end{thm}
\begin{proof}
It suffices to check inequality \eqref{condition}, and therefore, by Lemma \ref{lem:lowergamma2}, it suffices to check that 
\[|\tau| \min \left\{1,\left|\tanh\left(\frac{\pi\tau}{2}\right)\right|\right \}> \left(\frac{D_m}{2\pi}\right)^{2\sigma-1}\frac{\zeta(1-\sigma)^2}{\Gamma(1-\sigma)^2}.\]
But this follows from the fact that \[|\tau|>\frac{1}{\tanh\left(\frac{\pi}{2} \left(\frac{D_m}{2\pi}\right)^{2\sigma-1}\frac{\zeta(1-\sigma)^2}{\Gamma(1-\sigma)^2}\right)}\left(\frac{D_m}{2\pi}\right)^{2\sigma-1}\frac{\zeta(1-\sigma)^2}{\Gamma(1-\sigma)^2}>\left(\frac{D_m}{2\pi}\right)^{2\sigma-1}\frac{\zeta(1-\sigma)^2}{\Gamma(1-\sigma)^2},\] and therefore  \[\left|\tanh\left(\frac{\pi\tau}{2}\right)\right|\geq \tanh\left(\frac{\pi}{2} \left(\frac{D_m}{2\pi}\right)^{2\sigma-1}\frac{\zeta(1-\sigma)^2}{\Gamma(1-\sigma)^2}\right).\]
\end{proof}

\begin{cor}\label{prop: noneffective}
 Let $s=\sigma+i\tau$ with $\sigma<0$. There exists $T_\sigma \in \R_{>0}$ such that the Northcott property holds at $s$ for any $B>0$ as long as $|\tau|>T_\sigma$. 
\end{cor}

\section{The left side neighborhood of zero} \label{sec:around0}

In the last section, Theorem \ref{thm: analytical_effective} provided an answer for the Northcott property for the region $\sigma_0 \leq \sigma < 0$ but failed to capture the true boundary of condition \eqref{condition} in Proposition \ref{suffnorthcott}. In particular, we expect that for some $\sigma\in [\sigma_0,0)$, the Northcott property holds for all $\tau$.  

Using numerical methods, we can obtain substantially more precise results that are closer to the boundary given by Proposition \ref{suffnorthcott}. Namely, we can describe a circle around $s=-1$ in the style of the circles described in Theorem \ref{propevencircles} and Proposition \ref{prop:oddcomplex}, and we can better describe the behavior for $\sigma<0$ approaching the origin. Furthermore, we can also obtain improvements for the circles in the first few cases of $s=n$ with $n$ a negative integer. The results in this section are better than the ones given in Theorem \ref{thm: analytical_effective} but require a large number of steps, and ultimately rely on the help of a computer.

The strategy in this section is the following. We will numerically construct an approximation of a curve that is close to the boundary of the red region in Figure \ref{fig:pointgraph1}, such that we will be able to guarantee that the Northcott property is true for $s=\sigma+i\tau$ with $|\tau|>t>0$ such that $\sigma +it$ is on the curve. This curve approximation will be made of small horizontal segments (see Figure \ref{fig:line_graph} for an example).  
 In order to construct this approximation, we consider an interval around $n$, and we perform a sufficiently fine division into smaller  intervals $[\alpha,\beta]$ where we can numerically control the behavior of the factors involved in inequality \eqref{condition} due to monotonicity. To achieve this goal, we will need some auxiliary results about the growth of the factors involved in \eqref{condition}. 

The first result will allow us to understand the growth of $|\Gamma(s)|$ as $s$ moves in a horizontal line, outside a circle of center $\frac{1}{2}$.
\begin{lem} \label{lem: gamma_circle}
Let $s= \sigma+i\tau \in \{z\in\C_{\sigma<0} : |z-0.5|\geq 1.1\}$. We have 
\begin{align*}
    \frac{d}{d\sigma} |\Gamma(1-s)|^2 \leq 0.
\end{align*}
\end{lem}
\begin{rem} The constant 1.1 has been numerically chosen by numerically adjusting a circle centered at 0.5 so that it satisfies the following conditions. 
\begin{itemize}
\item The circle  encapsulates the region of $\C_{\sigma<0}$ where   $\frac{d}{d\sigma} |\Gamma(1-s)|^2 \geq 0$.
\item The boundary of the red region of Figure \ref{fig:pointgraph1} entirely lies  outside this circle (see Figure \ref{Fig4-closer} for more detail). 
\end{itemize}
\begin{figure}
    \centering
    \includegraphics[scale = 0.6]{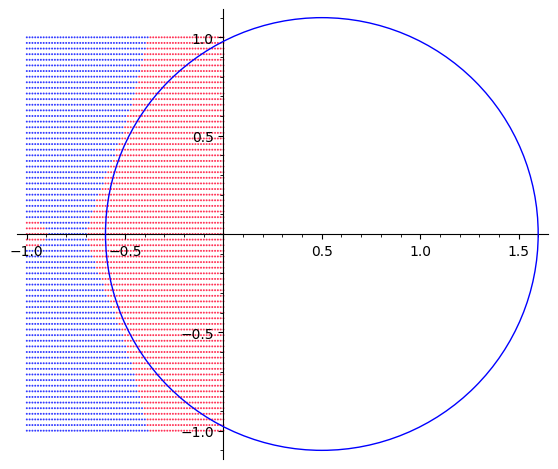}
    \caption{Depiction of the boundary of condition \eqref{condition} and the circle of center $0.5$ and radius $1.1$ employed in Lemma \ref{lem: gamma_circle}
    } 
    \label{Fig4-closer}
\end{figure}

\end{rem}

\begin{proof}

After replacing $1-s$ by $s$, the statement to be proven is equivalent to  $\frac{d}{d\sigma} |\Gamma(\sigma +i\tau)|^2 \geq 0$ for $s= \sigma+i\tau \in \mathcal{C}:=\{z\in\C_{\sigma>1} : |z-0.5|\geq 1.1\}$.  We have
\[\frac{d}{ds} \Gamma(s) \overline{\Gamma(s)}=(\psi(s)+\overline{\psi(s)}) |\Gamma(s)|^2,\]
where $\psi$ denotes the digamma function \eqref{eq:psi-defn}. Therefore, 
the desired derivative is given by  
\begin{align*}
    \frac{d}{d\sigma} |\Gamma(\sigma +i\tau)|^2 = 2\Re(\psi(\sigma+i\tau))|\Gamma(\sigma+i\tau)|^2.
\end{align*}

Our goal  is to show that $\Re(\psi(\sigma+i\tau))$ is positive outside the circle $\mathcal{C}$. We consider the series \eqref{eq:psi-series} and truncate it to the first four terms:
\begin{align}\label{eq:psibound}
    \Re(\psi(s)) = -\gamma + \sum_{k=1}^\infty \frac{k(\sigma-1)+(\sigma-1)^2+\tau^2}{k((k+\sigma-1)^2+\tau^2)} \geq -\gamma + \sum_{k=1}^4\frac{k(\sigma-1)+(\sigma-1)^2+\tau^2}{k((k+\sigma-1)^2+\tau^2)}, 
\end{align}
where we have chosen to keep four terms because numerical estimates suggest this level of truncation yields the needed precision. 

Next, since we want to prove that the right-hand side of \eqref{eq:psibound} is positive outside the circle ${\mathcal{C}}$, we consider polar coordinates centered at $0.5$. After the change of variables $s=0.5+re^{ix}$, we now have
\begin{align}\label{eq: truncated_polar}
    -\gamma + \sum_{k=1}^4 \frac{4r^2+4(k-1)r\cos(x)-2k+1}{4kr^2+4(2k^2-k)r\cos(x)+4k^3-4k^2+k}.
\end{align}
The derivative with respect to $r$ of each term is 
\[ \frac{4 \left(4 k(k-1) \cos\left(x\right) + (4 r^{2}+1) \cos\left(x\right) + 4r(2 k - 1)\right) }{\left(  4 r^{2} +4(2k-1) r \cos\left(x\right)+ 4k^2- 4 k + 1\right)^2}.\]
We see that the term above is  always positive for $x\in[-\pi/2,\pi/2]$. 
Thus, if expression \eqref{eq: truncated_polar} is positive on this  half circle of fixed radius, the result will follow.  To see this, we compute  the derivative of \eqref{eq: truncated_polar} with respect to $x$ and remark that it can be expressed as  $\sin(x) R(\cos(x))$, where $R(t)$ is a rational function with positive coefficients, and therefore $R(\cos(x))\geq 0$ for 
$x \in [-\pi/2,\pi/2]$.  
From this we can deduce that for fixed $r$, equation \eqref{eq: truncated_polar} reaches its minimum when $x=0$. Now we evaluate at $r=1.1$ and $x=0$ to obtain 
\begin{align*}
    -\gamma + \sum_{k=1}^4 \frac{4(1.1)^2+4(k-1)(1.1)-2k+1}{4k(1.1)^2+4k^3+4(2k^2-k)(1.1)-4k^2+k} = 0.00133\dots. 
\end{align*}
Thus, we conclude that $\re(\psi(s))\geq 0$ outside $\mathcal{C}$, and therefore the same is true for $\frac{d}{ds}|\Gamma(s)|^2$.
\end{proof}
The next result allows us to understand the growth of $\gamma_\C(s)$ and $\gamma_\R(s)$ as $s$ moves along a vertical line.  
\begin{lem}\label{lem: lowertau}
Let $\sigma\leq 1/2$ be fixed. Then
\begin{align*}
    \frac{d}{d\tau} \gamma_\C(s)^4 
    \geq 0 \text{ and } \frac{d}{d\tau} \gamma_\R(s)^2  
    \geq 0
\end{align*}
for $\tau \in [0,\infty)$.
\end{lem}
\begin{proof}
Starting with the case of $\gamma_\C$, it suffices to check the sign of the derivative. Thus, we can ignore the positive constants and consider
\begin{align*}
    \frac{d}{d\tau}\frac{|\Gamma(1-s)|^2}{|\Gamma(s)|^2}.
\end{align*}
The above expression has the same sign as
\begin{align*}
    &\left(\frac{d}{d\tau}\Gamma(1-s)\overline{\Gamma(1-s)}\right)|\Gamma(s)|^2 - \left(\frac{d}{d\tau}\Gamma(s)\overline{\Gamma(s)}\right)|\Gamma(1-s)|^2\\
    =& |\Gamma(1-s)|^2\left(-i\psi(1-s)+i\overline{\psi(1-s)}\right)|\Gamma(s)|^2 - |\Gamma(s)|^2\left(i\psi(s)-i\overline{\psi(s)}\right)|\Gamma(1-s)|^2.
\end{align*}
Ignoring positive terms once again, we are left to consider
\begin{align*}
    \left(-i\psi(1-s)+i\overline{\psi(1-s)}\right) - \left(i\psi(s)-i\overline{\psi(s)}\right),
\end{align*}
which is the same as
\begin{align*}
    \im(\psi(s)+\psi(1-s)).
\end{align*}
Now, using \eqref{eq:psi-series}, we reduce the problem to 
\begin{align*}
    \sum_{k=0}^{\infty} \frac{\tau}{(k+\sigma-1)^2+\tau^2} - \sum_{k=0}^{\infty} \frac{\tau}{(k-\sigma)^2+\tau^2} \geq 0.
\end{align*}
We can show this by ignoring the sums and comparing the terms for each $k$. In doing so, we find
\begin{align*}
     \frac{\tau}{(k+\sigma-1)^2+\tau^2}   \geq \frac{\tau}{(k-\sigma)^2+\tau^2} \text{ when } \sigma \leq 1/2 \text{ and } \tau \geq 0.
\end{align*}
The case of  $\gamma_\R$ can be proven similarly.
\end{proof}

\begin{cor}
Suppose that for fixed $t\geq 0$ and all $\sigma \in[\alpha,\beta]$ we have
    $$\frac{\Gamma_m(\sigma+it)}{\zeta(1-\sigma)} D_m^{\frac{1}{2}-\sigma}> 1.$$
Then, this inequality is also true for $\tau>t$.
\end{cor}
\begin{proof}
The statement follows directly from Lemma \ref{lem: lowertau}. 
\end{proof}

We are now ready to show the following key result.
\begin{prop} \label{prop:taucondition}
Let $s= \sigma+i\tau \in \{z\in\C_{\sigma<0} : |z-0.5|\geq 1.1\}$. Furthermore, let $\tau$ be fixed and $\sigma \in [\alpha,\beta] \subseteq\R$ and $n \in \Z_{<0}$.  The following statements are true:
\begin{enumerate}
    \item Suppose $[\alpha,\beta] \subseteq  [n,n+\frac{1}{2}]$, then for all $\sigma \in [\alpha,\beta]$
    \begin{align}\label{gamC_inc_dec}
     \frac{1}{\sqrt{\pi}} \left( \frac{D_m}{2\pi} \right)^{\frac{1}{2}-\beta} \frac{|\Gamma(1-\beta-i\tau)|}{\zeta(1-\beta)} \left|\sin(\pi (\alpha +i\tau))\right|^{\frac{1}{2}} > 1 \Longrightarrow \frac{\gamma_\C(s)}{\zeta(1-\sigma)} D_m^{\frac{1}{2}-\sigma}> 1.
    \end{align}
    \item Suppose $[\alpha,\beta] \subseteq  [n-\frac{1}{2},n]$, then
    \begin{align}\label{gamC_dec_dec}
     \frac{1}{\sqrt{\pi}} \left( \frac{D_m}{2\pi} \right)^{\frac{1}{2}-\beta} \frac{|\Gamma(1-\beta-i\tau)|}{\zeta(1-\beta)} \left|\sin(\pi (\beta +i\tau))\right|^{\frac{1}{2}} > 1 \iff \frac{\gamma_\C(s)}{\zeta(1-\sigma)} D_m^{\frac{1}{2}-\sigma}> 1.
    \end{align}
    \item Suppose $[\alpha,\beta] \subseteq  [2n,2n+1]$, then
    \begin{align*}
     \frac{\sqrt{2}}{\sqrt{\pi}} \left( \frac{D_m}{2\pi} \right)^{\frac{1}{2}-\beta} \frac{|\Gamma(1-\beta-i\tau)|}{\zeta(1-\beta)} \left|\sin(\frac{\pi}{2} (\alpha+i\tau))\right|> 1 \Longrightarrow \frac{\gamma_\R(s)}{\zeta(1-\sigma)} D_m^{\frac{1}{2}-\sigma}> 1.
    \end{align*}
    \item Suppose $[\alpha,\beta] \subseteq  [2n-1,2n]$, then
    \begin{align*}
     \frac{\sqrt{2}}{\sqrt{\pi}} \left( \frac{D_m}{2\pi} \right)^{\frac{1}{2}-\beta} \frac{|\Gamma(1-\beta-i\tau)|}{\zeta(1-\beta)} \left|\sin(\frac{\pi}{2} (\beta+i\tau))\right|> 1 \iff \frac{\gamma_\R(s)}{\zeta(1-\sigma)} D_m^{\frac{1}{2}-\sigma}> 1.
    \end{align*}
\end{enumerate}
\end{prop}
\begin{proof}
The proof follows by studying the growth of the terms in condition \eqref{condition}.

 The functions under consideration are (up to a positive constant) the following. 
\begin{enumerate}
    \item The expression 
    \begin{align*}
        \left(\frac{D_m}{2\pi}\right)^{\frac{1}{2}-\sigma}\frac{|\Gamma(1-\sigma-i\tau)|}{\zeta(1-\sigma)}
    \end{align*}
    is monotonously decreasing as a function of $\sigma$ when $s \in  \{z\in\C_{\sigma<0} : |z-0.5|\geq 1.1\}$  by Lemma \ref{lem: gamma_circle}.
    \item The function $|\sin(\pi\sigma)|$ is increasing in $[n,n+\frac{1}{2}]$ and decreasing in $[n-\frac{1}{2},n]$ for $n \in \Z$. For fixed $\tau$ this function has the same increasing/decreasing intervals as $|\sin(\pi (\sigma+i\tau))|^{\frac{1}{2}}$. This follows from the fact that  $|\sin(\pi (\sigma+i\tau))|^{\frac{1}{2}}= \sqrt[4]{\sin(\pi\sigma)^2+\sinh(\pi\tau)^2}$ and $\sinh(\pi\tau)^2$ is constant when $\tau$ is fixed. 
    \item Similarly, $\left|\sin(\frac{\pi}{2}(\sigma+i \tau))\right|$ is increasing in $[2n,2n+1]$ and decreasing in $[2n-1,2n]$ for $n \in \Z$.
\end{enumerate}

We get the result by combining the above facts together and by taking the minimum of each component in the interval $[\alpha,\beta]$.
\end{proof}

\subsection{An application of Proposition \ref{prop:taucondition}}

We can construct an effective approximation of the boundary of condition \eqref{condition}  by considering a very thin subdivision of the interval of interest and by applying Proposition \ref{prop:taucondition} in each sub-interval. To illustrate this, Figures  \ref{fig:line_graph} and  \ref{fig:line_graph_near_0_5} have been constructed by subdividing the intervals  $[-1.2,-0.8]$ and $[-0.8,-0.5]$  respectively. 

To obtain these results, we first subdivide the interval $[-1.5,-0.1]$ into segments of length $\delta = 0.0025$. For each said segment $[\alpha, \alpha+\delta]$, we then find $\tau_\alpha$ such that Proposition \ref{prop:taucondition} is satisfied. Since Lemma \ref{lem: lowertau} implies that Proposition \ref{prop:taucondition} is also satisfied for $|\tau|>|\tau_\alpha|$, we seek to find the smallest $\tau_\alpha$ possible. In the case of the sub-interval $[-1.5,-0.68]$, the region where $\tau_\alpha>0$ is close to a circle,  and therefore we approximate its boundary by above with a circle. 
We showcase these results in the following remark.
\begin{rem}\label{rem: numproof}
Let $s= \sigma+i\tau$ with $\sigma\in[-1.5,\sigma_1]$, where $\sigma_1\approx -0.68$ satisfies 
\[\frac{(2e^\gamma)^{\frac{1}{2}-\sigma_1}}{\zeta(1-\sigma_1)}\left|\frac{\Gamma(1-\sigma_1)}{\Gamma(\sigma_1)}\right|^\frac{1}{2}=1.\]
(In other words, $\frac{\gamma_\C(\sigma_1)}{\zeta(1-\sigma_1)}D_m^{\frac{1}{2}-\sigma_1}=1$.)
We define
\begin{align*}
    R(\sigma) =
    \begin{cases}
        0 & \text{for } \sigma \in [-1.5,-1 - \rho(-1))\\
       \sqrt{\rho(-1)^2-(\sigma+1)^2} & \text{for } \sigma \in [-1 - \rho(-1),-1 + \rho(-1)]\\
        0 & \text{for } \sigma \in (-1+\rho(-1),\sigma_1]\\
    \end{cases}
\end{align*}
where $\rho(-1) \approx \num{9.5e-2}$.
Then, as long as $|\tau|\geq R(\sigma)$ the Northcott property holds at $s$ for any $B > 0$. 
\end{rem}
We stop at $\sigma_1\approx -0.68$ because for $\sigma>\sigma_1$ we will again need that $|\tau_\alpha|>0$ for Proposition \ref{prop:taucondition} to be satisfied. Figure \ref{fig:line_graph_near_0_5} illustrates this phenomenon. The value of $\sigma_1$ marks the beginning of the red region on the right of Figure \ref{fig:pointgraph1}. 
The complete list of $\tau_\alpha$ used for the interval $[-1.5,\sigma_1]$ is available here  \cite{Genereux_The-northcott-property-of-dedekind-zeta-functions_2022}. Notice that in this list, the closest point to $s=0.5$ is at a distance of $1.1227 > 1.1$, in particular, this justifies our use of Proposition \ref{prop:taucondition}.
Figure \ref{fig:line_graph_56} illustrates a comparison between Theorem \ref{thm: analytical_effective}, Remark \ref{rem: numproof}, and the numerical graph found with Proposition \ref{prop:taucondition}  and more clearly depicted in Figure \ref{fig:line_graph}.

For the interval $[\sigma_1,-0.1]$, the boundary of the Northcott region is more difficult to describe. The following remark aims at approximating the boundary described by bounding the segments by above with continuous functions. A comparison among Theorem \ref{thm: analytical_effective}, Remark \ref{rem: numproof2}, and the numerical results found with Proposition \ref{prop:taucondition}  is made in Figure \ref{fig:line_graph_57}.

\begin{rem}\label{rem: numproof2}

Let $s= \sigma+i\tau$ with $\sigma\in[\sigma_1,-0.1]$. We define
\begin{align*}
    R(\sigma) =
    \begin{cases}
        \sqrt{\sigma-\sigma_1+0.1} & \text{for } \sigma \in (\sigma_1,-0.65]\\
         0.82\left(\frac{D_m}{2\pi}\right)^{2\sigma-1}\frac{\zeta(1-\sigma)^2}{\Gamma(1-\sigma)^2} & \text{for } \sigma \in (-0.65,-0.1].
    \end{cases}
\end{align*}
Then, as long as $|\tau|\geq R(\sigma)$ the Northcott property holds at $s$ for any $B > 0$.

The formula for $R(\sigma)$ was constructed ad hoc from a modification of the formula in Theorem \ref{thm: analytical_effective} aimed at approximating the numerical graph from Figure \ref{fig:line_graph_near_0_5} (from above) combined with a simpler formula for values closer to $\sigma_1$. 
\end{rem}
We have chosen to stop at $-0.1$ for the sake of  clarity, as $\tau_\alpha \to \infty$ as $\alpha\to 0$. This method could have been used for any interval $[-1.5,\varepsilon]$ where $\varepsilon>0$.

\begin{rem}

Although the areas around the integers that do not satisfy the condition of Proposition $\ref{suffnorthcott}$ are not circles, they are approximate circles well enough suggesting that one should fit them in circles centered at the integers. Table \ref{tab: values} compares this numerical method based on Proposition \ref{prop:taucondition} with the previous results based on Theorem
\ref{propevencircles} and Propositions \ref{lem:oddreal} and \ref{prop:oddcomplex}. We have used better precision than in Remarks \ref{rem: numproof} and \ref{rem: numproof2}, leading to more precision in the radii.

 The source code for these calculations is available here \cite{Genereux_The-northcott-property-of-dedekind-zeta-functions_2022}. The computations were performed with the aid of Sage \cite{sagemath}.

\end{rem}

\begin{figure}
    \centering
    \includegraphics{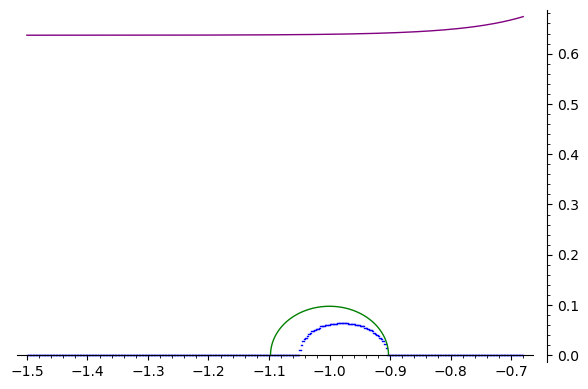}
    \caption{Comparison between Theorem \ref{thm: analytical_effective} (purple), Remark \ref{rem: numproof} (green) and the piece-wise curve found
    with the methods of this section (blue). }
    \label{fig:line_graph_56}
\end{figure}

\begin{figure}
    \centering
    \includegraphics{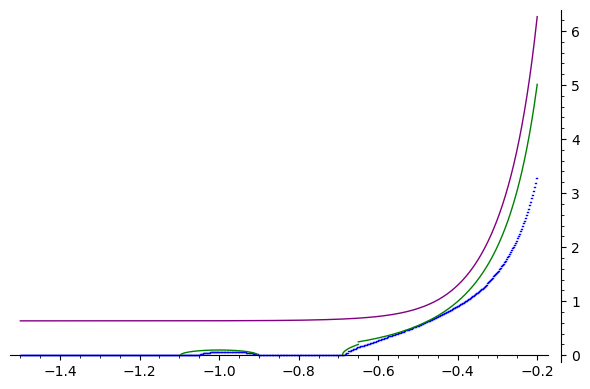}
    \caption{Comparison between Theorem \ref{thm: analytical_effective} (purple), Remark \ref{rem: numproof2} (green) and the piece-wise curve found
    with the methods of this section (blue). 
    }
    \label{fig:line_graph_57}
\end{figure}

\begin{figure}
    \centering
    \includegraphics{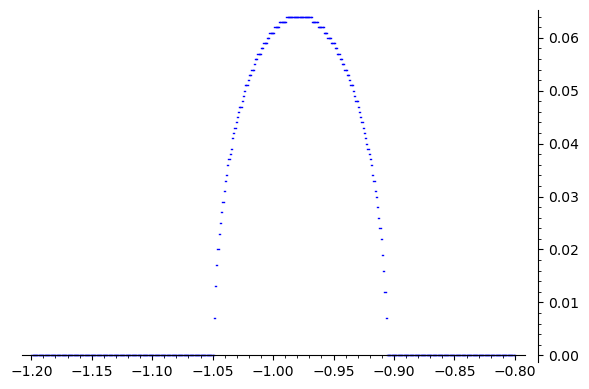}
    \caption{Computer assisted boundary for the interval $[-1.2,-0.8]$.  Everything above the blue line is Northcott.}
    \label{fig:line_graph}
\end{figure}

\begin{figure}
    \centering
    \includegraphics{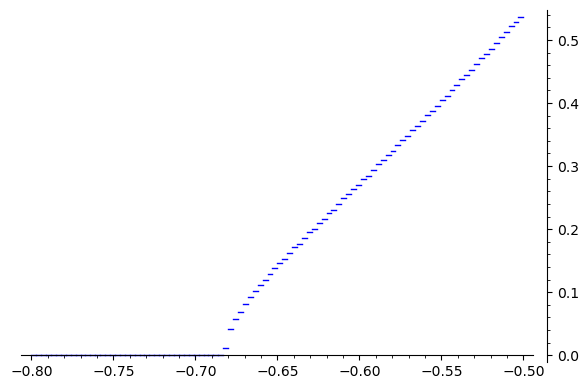}
    \caption{Computer assisted boundary for the interval $[-0.8,-0.5]$. Everything above the blue line is Northcott.}
    \label{fig:line_graph_near_0_5}
\end{figure}

\begin{table}
\begin{tabular}{|c|c|c|c|}
 \hline
 Center & Num. radius (Section \ref{sec:around0}) & Section \ref{sec:left} radius & Non-Northcott radius\\ [0.5ex] 
 \hline\hline
 \\[-1em]
 $-1$ &  \num{9.260260274818e-2} & - & \num{3.415443142941e-6} \\ 
 \hline
 \\[-1em]
 $-2$ &  \num{2.105502084026e-2} & \num{6.388919396319e-2} & \num{1.330026824001e-8} \\
 \hline
 \\[-1em]
 $-3$ &  \num{4.474651495645e-6} & \num{5.742868294706e-5} & \num{9.877567910286e-12} \\
 \hline
 \\[-1em]
 $-4$ &  \num{1.135531168473e-4} & \num{4.516050376141e-4} & \num{3.572719521466e-15} \\
 \hline 
 \\[-1em]
 $-5$ &  \num{6.138786399296e-11} & \num{1.190762805871e-9} & \num{8.022539291403e-19} \\ 
 \hline
\end{tabular}\medskip \caption{ \label{tab: values} In the first column, we showcase the radii of the Northcott region obtained using the methods of this section. This means that the points outside of these circles satisfy the Northcott property. The second column is similarly constructed using the methods given by Theorem \ref{propevencircles} and Propositions \ref{lem:oddreal} and \ref{prop:oddcomplex} to calculate the radii. Finally, the third column shows the radii of the non-Northcott circles, also computed from the results in Section \ref{sec:left}. In this case, the area inside these circles is proven to be non-Northcott. 
}
\end{table}

\section{Inside the critical strip} \label{sec:inside}
\subsection{The case of $s=1$}

This case was established in \cite{PP} by using the asymptotics for the moment of the class numbers $h_{\Q(\sqrt{D})}$, which are directly related to $\zeta^*_{\Q(\sqrt{D})}(1)$ by the class number formula. 

Another strategy is to employ the following result of Chowla. 
\begin{thm}\cite[Theorem 2]{Chowla} For any $x$ sufficiently large, there is a $x<d<2x$ such that there is a real primitive character $\chi_d$ modulo $d$ satisfying \[L(1,\chi_d)\leq (1+o(1))\frac{\zeta(2)}{e^\gamma \log (\log d)},\]
where $\gamma$ denotes the Euler--Mascheroni constant.
\end{thm}

We then have the following result.
\begin{thm}
$s=1$ does not satisfy the Northcott property for any $B>0$. 
\end{thm}
\begin{proof}
Recall that for a quadratic field $K$ with corresponding quadratic character $\chi$ we have 
\[\zeta_K^*(1)=\lim_{s\rightarrow 1}(s-1)\zeta_K(s)=\lim_{s\rightarrow 1}(s-1)\zeta(s)L(1,\chi)=L(1,\chi).\]

Given $B>0$, choose $x$ such that \[ (1+o(1))\frac{\zeta(2)}{e^\gamma \log(\log x )}\leq B.\]
By choosing $x$ progressively larger, we can construct an infinite sequence of  $D_k$ such that 
\[|\zeta_{\Q(\sqrt{D_k})}^*(1)| \leq B,\]
and this process can be applied to any $B$ arbitrary. Therefore, the Northcott property is not satisfied in this case.

\end{proof}

\subsection{The case of $1/2<\sigma<1$}

Let $\chi_p=\left(\frac{\cdot}{p} \right)$ denote the Legendre symbol modulo a prime $p$.  We have the following result of Lamzouri. 
\begin{thm}\cite[Theorem 1.8, partial statement]{Lamzouri} Assume the Generalized Riemann Hypothesis. Let $s =\sigma+i\tau$ where $1/2<\sigma<1$ and $\tau \in \R$. Let $x$ be large. Then there are $\gg x^\frac{1}{2}$ primes $p\leq x$ such that 
\[\log |L(s,\chi_p)| \leq -(\beta(s)+o(1)) \frac{(\log x)^{1-\sigma}}{ (\log(\log x))^{\sigma}},\]
where \[\beta(\sigma)=\frac{(2\log 2)^{\sigma-1}}{1-\sigma},\]
and for $\tau \not = 0$,  \[\beta(s)=\frac{\beta(\sigma)\tau^2}{(1-\sigma)^2+4\tau^2}.\]
\end{thm}

With this result we can prove the following. 
\begin{thm}\label{thm:segment} Assume the Generalized Riemann Hypothesis.
Let $s=\sigma+i\tau$ with  $1/2<\sigma<1$, then $s$ does not satisfy the Northcott property for any $B>0$. 
\end{thm}
\begin{proof}
Let $B>0$ and $s$ fixed with $1/2<\sigma<1$. Choose $x$ such that 
\[ |\zeta(s)| \exp\left(-(\beta(s)+o(1)) \frac{(\log x)^{1-\sigma}}{ (\log(\log x))^{\sigma}} \right)\leq B.\]
Thus, there are $\gg x^\frac{1}{2}$ primes $p\leq x$ such that 
\[|\zeta_{\Q(\sqrt{p})}(s)|\leq B.\]

By choosing $x$ progressively larger, we can construct an infinite sequence of primes $p_k$ such that 
\[|\zeta_{\Q(\sqrt{p_k})}(s)| \leq B,\]
and this process can be applied to any $B$ arbitrary. Therefore, the Northcott property is not satisfied in this case. 
\end{proof}

We can also give an unconditional partial result. For this we need  the following statement of Sono. 
\begin{thm}\cite[Theorem 2.2, simplified version]{Sono} \label{thm:Sono}
Let $\alpha_1, \alpha_2 \in \C$ such that $|\re(\alpha_j)|<1/2$.  Let $\Phi:\R_{>0}\rightarrow \R$ be a smooth function supported in $[1/2,3]$ and, for sufficiently large $X>0$, set \[F(x)=\Phi\left(\frac{x}{X}\right).\]
Let $\chi_{8d}=\left(\frac{8d}{\cdot}\right)$, where $d$ is square-free and odd ($8d$ is a fundamental discriminant). Then, for any $\varepsilon>0$, 
\[\sum_{\substack{d\, \square-\text{free}\\\text{odd}}}L\left(\textstyle{\frac{1}{2}}+\alpha_1, \chi_{8d}\right)
L\left(\textstyle{\frac{1}{2}}+\alpha_2, \chi_{8d}\right)F(d) =\sum_{\varepsilon_1,\varepsilon_2\in \{\pm 1\}} A_{\varepsilon_1\alpha_1, \varepsilon_2\alpha_2}\Gamma_{\alpha_1}^{\delta_1}\Gamma_{\alpha_2}^{\delta_2} \frac{2\tilde{F}(1-\delta_1\alpha_1-\delta_2\alpha_2)}{3\zeta(2)}+O_{\alpha_1,\alpha_2}\left(X^{\frac{1}{2}+\varepsilon}\right),\]
where $\tilde{F}(w)=\int_0^\infty F(x) x^{w-1} dx$ is the Mellin transform of $F$, 
\[A_{\alpha_1,\alpha_2}=\sum_{(n,2)=1}\frac{\sigma_{\alpha_1,\alpha_2}(n^2)}{n}\prod_{p\mid n} (1+p^{-1})^{-1},\]
and 
\[\sigma_{\alpha_1,\alpha_2}(n)=\sum_{n_1n_2=n} \frac{1}{n_1^{\alpha_1} n_2^{\alpha_2}}.\]
We also have 
\[\delta_i=\begin{cases}0 & \varepsilon_i=1,\\
1 & \varepsilon_i=-1.\end{cases}\]
Finally, 
\[\Gamma_\alpha=\left(\frac{8}{\pi}\right)^{-\alpha}\frac{\Gamma\left(\frac{1-2\alpha}{4}\right)}{\Gamma\left(\frac{1+2\alpha}{4}\right)}.\]
\end{thm}
The statement of Theorem \ref{thm:Sono} is deduced from the recipe of Conrey, Farmer, Keating, Rubinstein, and Snaith \cite{CFKRS} and rigorously proven as an intermediate step in the estimate  for the second moment of quadratic Dirichlet $L$-functions at $s=1/2$ with a square-root savings in the error term. The original ideas for the proof are due to Young, who developed this method to obtain similar results for the first and third moments in \cites{Young-first,Young-third}. 
In its original statement, Theorem \ref{thm:Sono} has the condition that the $\alpha_j$ be in the rectangle $|\re(s)|\leq \frac{\varepsilon}{\log X}$, $|\im(s)|\leq X^\varepsilon$. However, as explained by Young in \cite{Young-first}, this condition is imposed to claim uniformity in terms of $\alpha$. In our case, since the $\alpha_j$ will be fixed, we do not need to impose this condition, but we still need $|\re(\alpha_j)|<1/2$ for all the terms $A_{\varepsilon_1\alpha_1, \varepsilon_2\alpha_2}$ to converge. 

We choose a concrete $\Phi(x)$ in the next statement to make the result explicit. 
\begin{thm} \label{thm:Phi}
Let $s \in \C$ such that $1/2<\sigma<1$ be fixed. Set 
\[ \Phi(x)=\begin{cases}\exp\left( -\frac{1}{(2x-1)(3-x)}\right) & \frac{1}{2}<x< 3,\\0 & \text{otherwise},   \end{cases} \mbox{ and } I=\int_0^\infty \Phi(y) dy.\]

Then $s$ does not satisfy the Northcott property for any 
\begin{equation}\label{eq:B}B > |\zeta\left(\textstyle{s}\right) |\left( \exp\left(\frac{1}{2}\right)  A_{s-\frac{1}{2}, \overline{s}-\frac{1}{2}} \frac{8I}{9}\right)^\frac{1}{2}.\end{equation}

\end{thm}

\begin{proof} Let $\alpha=s-\frac{1}{2}$. Then $0<\re(\alpha)<1/2$, and we set  $\alpha_1=\overline{\alpha_2}=\alpha$ in Theorem \ref{thm:Sono}. We have
\[\tilde{F}(1-\delta_1\alpha-\delta_2\overline{\alpha})=\int_0^\infty\Phi\left(\frac{x}{X}\right) x^{-\delta_1\alpha-\delta_2\overline{\alpha}} dx=X^{1-\delta_1\alpha-\delta_2\overline{\alpha}} \int_0^\infty\Phi\left(y\right) y^{-\delta_1\alpha-\delta_2\overline{\alpha}} dy.\]
Remark that the dominant term in the formula for the mixed moment given in Theorem \ref{thm:Sono} occurs with $\delta_i=0$. This gives
\[\sum_{\substack{d\, \square-\text{free}\\\text{odd}}}\left|L\left(\textstyle{\frac{1}{2}}+\alpha, \chi_{8d}\right)\right|^2
F(d) = A_{\alpha, \overline{\alpha}} \frac{2I}{3\zeta(2)}X+O\left(X^{1-\re(\alpha)}+X^{\frac{1}{2}+\varepsilon}\right).\]

Then we get 
\[\sum_{\substack{d\, \square-\text{free}\\\text{odd}\\ X\leq d \leq \frac{5}{2}X }}\left|L\left(\textstyle{\frac{1}{2}}+\alpha, \chi_{8d}\right)\right|^2\leq \exp\left(\frac{1}{2}\right)  A_{\alpha, \overline{\alpha}} \frac{2I}{3\zeta(2)}X+O\left(X^{1-\re(\alpha)}+X^{\frac{1}{2}+\varepsilon}\right).\]
The number of $d$ square-free, odd and such that  $X\leq d \leq \frac{5}{2}X$ is $\sim \frac{3X}{4\zeta(2)}$ (see \cite[Ex. 3.2.1.6]{Montgomery-Vaughan}).

Thus, given $\varepsilon>0$, we can guarantee that for $X$ large enough there is a $d$ such that $X \leq d\leq \frac{5}{2}X$ and for which 
\[\left|L\left(\textstyle{\frac{1}{2}}+\alpha, \chi_{8d}\right)\right|\leq \left( \exp\left(\frac{1}{2}\right)  A_{\alpha, \overline{\alpha}} \frac{8I}{9}\right)^\frac{1}{2}+\varepsilon.\]
Taking $\varepsilon$ arbitrarily small we can construct an infinite sequence of $d$'s satisfying this property, and leading to bounded $\left|L\left(\textstyle{\frac{1}{2}}+\alpha, \chi_{8d}\right)\right|$. The conclusion follows by writing 
\[\zeta_K\left(\textstyle{\frac{1}{2}}+\alpha, \chi_{8d} \right)=\zeta\left(\textstyle{\frac{1}{2}}+\alpha, \chi_{8d}\right)L\left(\textstyle{\frac{1}{2}}+\alpha, \chi_{8d}\right).\]
\end{proof}

\begin{rem}
Notice that 
\begin{align*}
& \frac{1}{2}\left( \left(1-\frac{1}{p^{\frac{1}{2}+\alpha}}\right)^{-1}\left(1-\frac{1}{p^{\frac{1}{2}+\overline{\alpha}}}\right)^{-1}+\left(1+\frac{1}{p^{\frac{1}{2}+\alpha}}\right)^{-1}\left(1+\frac{1}{p^{\frac{1}{2}+\overline{\alpha}}}\right)^{-1}\right)\\
&=\frac{1}{2}\left( \sum_{j_1=0}^\infty \frac{1}{p^{j_1(\frac{1}{2}+\alpha)}} \sum_{j_2=0}^\infty \frac{1}{p^{j_2(\frac{1}{2}+\overline{\alpha})}}+\sum_{j_1=0}^\infty \frac{(-1)^{j_1}}{p^{j_1(\frac{1}{2}+\alpha)}} \sum_{j_2=0}^\infty \frac{(-1)^{j_2}}{p^{j_2(\frac{1}{2}+\overline{\alpha})}}\right) \\
&=1+\frac{1}{2}\left(\sum_{\ell=1}^\infty\frac{ \sigma_{\alpha,\overline{\alpha}}(p^{\ell})}{p^\frac{\ell}{2}} +\sum_{\ell=1}^\infty\frac{(-1)^\ell \sigma_{\alpha,\overline{\alpha}}(p^{\ell})}{p^\frac{\ell}{2}}
\right).
\end{align*}
Thus, $ A_{\alpha,\overline{\alpha}}$ can be expressed as \begin{align*}
 A_{\alpha,\overline{\alpha}}=&\prod_{p\not= 2}\left(1 + \left(1+\frac{1}{p}\right)^{-1} \sum_{\ell=1}^\infty \frac{\sigma_{\alpha, \overline{\alpha}}(p^{2\ell})}{p^\ell}\right)\\
 =& \prod_{p\not = 2}
 \left[\frac{1}{2}\left( \left(1-\frac{1}{p^{\frac{1}{2}+\alpha}}\right)^{-1}\left(1-\frac{1}{p^{\frac{1}{2}+\overline{\alpha}}}\right)^{-1}+\left(1+\frac{1}{p^{\frac{1}{2}+\alpha}}\right)^{-1}\left(1+\frac{1}{p^{\frac{1}{2}+\overline{\alpha}}}\right)^{-1}\right)+\frac{1}{p}\right]\\ &\times \left(1+\frac{1}{p}\right)^{-1}.
\end{align*}

\end{rem}

\bibliographystyle{amsalpha}
\bibliography{Bibliography}

\end{document}